\documentclass[a4paper]{article}
\usepackage[T1]{fontenc}
\usepackage{amsmath,amssymb,amsthm,mathrsfs}
\usepackage{cite}
\usepackage{mathtools}
\usepackage{enumitem}

\usepackage{xcolor}

\usepackage[pdftex,bookmarksopen=true,bookmarks=true,unicode,setpagesize]{hyperref}
\hypersetup{colorlinks=true,linkcolor=black,citecolor=black,urlcolor=blue}

\let\equation=\gather
\let\endequation=\endgather

\allowdisplaybreaks[3]
\numberwithin{equation}{section}

\makeatletter
\renewcommand*{\@fnsymbol}[1]{\ensuremath{\ifcase#1\or 1\or 2\or 3\else\@ctrerr\fi}}
\makeatother

\newcounter{exmpcounter}
\newtheorem{thm}{Theorem}[section]

\newtheorem{lem}[thm]{Lemma}
\newtheorem{prop}[thm]{Proposition}
\theoremstyle{definition}
\newtheorem{defn}[thm]{Definition}
\newtheorem{exmp}[exmpcounter]{Example}
\theoremstyle{remark}
\newtheorem{rem}[thm]{Remark}

\newcounter{assum}
\newenvironment{assum}[1][]{\ifx\newenvironment#1\newenvironment\refstepcounter{assum}\fi\equation\tag{\ensuremath{\mathrm{A}\theassum#1}}}{\endequation}

\DeclareMathOperator*{\esssup}{esssup}
\DeclareMathOperator*{\essinf}{essinf}

\newcommand{\R}{{\mathbb{R}}}
\newcommand{\X}{{\R^d}}
\renewcommand{\S}{{S^{d-1}}}
\newcommand{\N}{\mathbb{N}}
\newcommand{\eps}{\varepsilon}

\newcommand{\La}{\Lambda}
\newcommand{\ka}{\varkappa}
\newcommand{\kam}{\varkappa^-}
\newcommand{\x}{\mathcal{X}}
\newcommand{\lt}{l_\theta}
\newcommand{\Tauin}{\mathscr{C}}
\newcommand{\m}{{\mathfrak{m}}}
\newcommand{\n}{{\mathfrak{n}}} 
\newcommand{\1}{1\!\!1}
\newcommand{\inter}{{\mathrm{int}}}
\newcommand{\supp}{\mathrm{supp}\,}
\newcommand{\Cb}{C_b(\X)}
\newcommand{\Buc}{C_{ub} (\X)}
\newcommand{\Y}{\mathcal{U}}
\newcommand{\locun}{\xRightarrow{\,\mathrm{loc}\ }}

\title{The hair-trigger effect for a class of nonlocal nonlinear equations}

\author{Dmitri Finkelshtein\thanks{Department of Mathematics,
Swansea University, Singleton Park, Swansea SA2 8PP, U.K. ({\tt d.l.finkelshtein@swansea.ac.uk}).} \and Pasha Tkachov\thanks{Gran Sasso Science Institute, Viale Francesco Crispi, 7, 67100 L'Aquila AQ, Italy ({\tt pasha.tkachov@gssi.it}).}}

\begin{document}

\maketitle

\begin{abstract}

We prove the hair-trigger effect for a class of nonlocal nonlinear evolution equations on $\X$ which have only two constant stationary solutions, $0$ and $\theta>0$. The effect consists in that the solution with an initial
condition non identical to zero converges (when time goes to $\infty$) to $\theta$ locally uniformly in $\X$. We find also sufficient conditions for existence, uniqueness and comparison principle in the considered equations.   

\textbf{Keywords:} hair-trigger effect, nonlocal diffusion, reaction-diffusion equation, front propagation, monostable equation, nonlocal nonlinearity, long-time behavior, integral equation

\textbf{2010 Mathematics Subject Classification:} 35B40, 35K57, 47G20, 45G10  
\end{abstract}

\section{Introduction}
We will deal with the following nonlinear nonlocal evolution equation on the Euclidean space $\X$, $d\geq1$:
\begin{equation}\label{eq:basic}
	\dfrac{\partial u}{\partial t}(x,t) = \ka (a*u)(x,t) -mu(x,t) -u(x,t)(Gu)(x,t)
\end{equation}
for $t>0$, $x\in\X$, with an initial condition $u(x,0)=u_0(x)$, $x\in\X$. Here 
$m, \ka>0$; $a$ is a nonnegative probability kernel on~$\X$, i.e. $0\leq a\in L^1(\X)$ and
\begin{equation}\label{eq:normalized}
\int_\X a(x)\,dx=1;
\end{equation}
$(a*u)(x,t)$ means the convolution (in~$x$)
between  $a$ and $u$, namely,
\begin{equation}\label{eq:defconv}
	(a*u)(x,t)=\int_\X a(x-y)u(y,t)dy;
\end{equation}
and $G$ is a mapping on a space of bounded on $\X$ functions.

We interpret $u(x,t)$ as a density of a population at the point $x\in\X$ at the moment of time $t\geq 0$.
The probability kernel $a=a(x)$ describes distribution of the birth of new individuals with constant intensity $\ka>0$.
Individuals in the population may also die either with the constant mortality rate $m>0$ or because of the competition, described by the density dependent rate $Gu$, where $G$ is an (in general, also nonlinear) operator on a space of bounded functions (cf. the discussion in \cite{RHD2011}). 

The equation \eqref{eq:basic} can be also rewritten in a reaction-diffusion form 
\begin{gather}
		\dfrac{\partial u}{\partial t}(x,t) = \ka (a*u)(x,t) - \ka u(x,t) + (Fu)(x,t),\label{eq:RDform}\\
\shortintertext{where}
		Fu := u(\ka-m-Gu) \label{eq:reaction}\\
\intertext{plays the role of the so-called reaction term, whereas}
Lu:=\ka (a*u)-\ka u \label{eq:nonlocdif}
\end{gather}
describes the non-local diffusion generator, see e.g. \cite{AMRT2010} (note that $L$ is also known as the generator of a continuous time random walk in $\X$ or of a compound Poisson process on $\X$). As a result, the solution $u$ to the equation \eqref{eq:RDform} may be interpreted as a density of a species which invades according to a nonlocal diffusion within the space $\X$ meeting a reaction $F$; see e.g. \cite{Fif1979,Mur2003,SK1997}.

Below, we restrict ourselves  to the case where \eqref{eq:basic} has two constant solutions $u\equiv 0$ and $u\equiv \theta>0$ only. The main aim of the present paper is to find sufficient conditions for the so-called \emph{hair-trigger effect}.
The latter means that, unless $u_0\equiv0$, the corresponding solution to \eqref{eq:basic} achieves an arbitrary chosen level between $0$ and $\theta$ uniformly on an arbitrary chosen domain of $\X$ after a finite time.
In~other words, $u(x,t)$ converges, as $t\to\infty$, locally uniformly in $x\in\X$ to the positive stationary solution $u\equiv \theta$.
The latter constant solution, therefore, is globally asymptotically stable in the sense of the topology of local uniform convergence. Therefore, the equation \eqref{eq:basic} appears of the so-called monostable type; cf. also Remark~\ref{rem:monostability} below.

Firstly, a reaction-diffusion equation of the form \eqref{eq:RDform} was considered in the seminal paper \cite{KPP1937} by Kolmogorov--Petrovsky--Piskunov (KPP). There, for the local reaction $Fu=f(u)=u(1-u)^2$ (that corresponds to $Gu=2u-u^2$ in \eqref{eq:reaction}; we set also here $\ka-m=1$), the equation \eqref{eq:RDform} was derived from a model for the dispersion
of a spatially distributed species. To analyze the model, the authors used a diffusion scaling, which led to the classical local diffusion generator $\ka \Delta u$ (for $d=1$) instead of $L$ in \eqref{eq:RDform}. Moreover, they proposed the method which covered more general local reactions $Fu=f(u)$ as well. We will say that such local reaction $F$ has the \emph{KPP-type} if $f:\R\to\R$ is Lipschitz continuous on $[0,\theta]$ and 
\begin{equation}\label{eq:KPPtypereaction}
f(0)=f(\theta)=0;\qquad f'(0)>0;\qquad 0<f(r)\leq f'(0)r,\quad r\in(0,\theta).
\end{equation}
In particular, the logistic reaction $f(u)=u(\theta-u)$, that corresponds to the identical mapping $Gu=u$ in \eqref{eq:reaction}, satisfies \eqref{eq:KPPtypereaction}. The corresponding model was considered early by Fisher \cite{Fis1937}, it described the advance of a favorable allele through a spatially distributed population. 
Note that the conditions for the mapping $G$ (and hence, by product, for the reaction $F$) which we postulate in Section~\ref{sec:assumptionsandresults} below are reduced, in the case of a local reaction $Fu=f(u)$, to \eqref{eq:KPPtypereaction} (see Example~\ref{ex:local_reaction} below).

Later, the significance of nonlocal terms in diffusion and/or reaction in \eqref{eq:RDform} was stressed by many authors, in particular, in ecology and population biology, see e.g. \cite{LMNC2003, BCGR2014, CLMH2003}; see also recent papers \cite{Ayd2018,NTY2017} where the importance and observed effects of  nonlocal interactions in biological models are discussed. 

A natural nonlocal analogue of the Fisher--KPP equation with the mentioned local reaction $f(u)=u(\theta-u)$ is the  equation \eqref{eq:RDform} with both nonlocal diffusion generator \eqref{eq:nonlocdif} and the linear nonlocal mapping
$Gu=\kam a^-*u$ in \eqref{eq:reaction}, where $\kam>0$, $0\leq a^-\in L^1(\X)$ with $\int_\X a^-(x)\,dx=1$, and the convolution is defined as in \eqref{eq:defconv} (see Example~\ref{ex:Gu_equals_aminus_conv_u} below). The corresponding equations \eqref{eq:basic}, or \eqref{eq:RDform}, similarly to the classical Fisher--KPP equation, may be obtained from different models. In~particular, for the case $\ka=\ka^-$, $a=a^-$, it was obtained, for $m=0$ in \cite{Mol1972a, Mol1972} from a model of simple epidemic, whereas, for $m>0$, it was derived in \cite{Dur1988} from a crabgrass model on the lattice $\mathbb{Z}^d$. For different kernels $a$ and $a^-$, the equation \eqref{eq:basic} appeared in \cite{BP1997} from a population ecology model; see also \cite{BP1999,DL2000} and the rigorous derivation of \eqref{eq:basic} in \cite{FM2004, FKK2011a}

More generally, a nonlocal analogue of the local KKP-type reaction $f(u)=u(\theta-u)^n$ is, naturally, the reaction
\begin{equation}\label{eq:nlocreaction}
Fu = \gamma_n u (\theta-a^-*u)^n,\quad n\in\N,
\end{equation}
with $a^-$ is as above and $\gamma_n>0$ (see Example~\ref{ex:nonlocal_general} below). Note also that the equation \eqref{eq:RDform} with the nonlocal diffusion \eqref{eq:nonlocdif} and a local KPP-type reaction $Fu=f(u)$ was considered in \cite{Sch1980} motivated by an analogy to Kendall's epidemic model \cite{Ken1965}.

The first (up to our knowledge) result about the hair-trigger effect described above, for a non-linear evolution equation with the local diffusion, was shown by Kanel \cite{Kan1964}, for the cases of the combustion and the Fisher--KPP reaction-diffusion equations in the dimension $d=1$. Multidimensional analogues were shown by Aronson and Weinberger \cite{AW1975,AW1978}; in the latter reference the notion `hair-trigger' was, probably, firstly used.

For the nonlocal diffusion \eqref{eq:nonlocdif}, the first result about the hair-trigger effect for a solution to \eqref{eq:RDform} was  obtained in \cite{LPL2005}: for the one-dimensional case $d=1$, under additional restrictions on the probability kernel $a=a(x)$, and for a local reaction  $Fu=f(u)$ of the KPP-type given by \eqref{eq:KPPtypereaction}.

For the nonlocal diffusion in $\X$ with $d>1$, the hair-trigger effect, for the local reaction term $f(u)=u^{1+p}(1-u)$ with $p>0$, has been shown recently in \cite{Alf2016}, under additional assumptions on $a=a(x)$ (in particular, its radial symmetry was assumed). From this, by comparison-type arguments, it might be possible to show the hair-trigger effect for a 
local KPP-type reaction $Fu=f(u)$ described by  \eqref{eq:KPPtypereaction}, provided that, additionally, $f'(\theta)<0$.

To the best of our knowledge, the present paper is the first one 
that shows the hair-trigger effect for non-local reactions. 
In particular, we allow the reaction \eqref{eq:nlocreaction} 
in \eqref{eq:RDform}--\eqref{eq:reaction}, provided that an appropriate comparison between $a$ and $a^-$ is assumed (see Examples~\ref{ex:Gu_equals_aminus_conv_u}--\ref{ex:nonlocal_general} below).

Another novelty of the present paper, even for the case of the local KPP-type reactions $Fu=f(u)$ given by \eqref{eq:KPPtypereaction} is that we allow general anisotropic probability kernels $a=a(x)$, $x\in\X$ (see Example~\ref{ex:local_reaction} below).  Note, that, however, we do not cover the local reaction $f(u)=u^{1+p}(1-u)$ with $p>0$, considered in \cite{Alf2016}.

For results about the hair-trigger effect in other types of non-local equations see also \cite{Die1978a}.

The hair-trigger effect is an important tool in the study of the long-time behavior of evolution equations.
In particular, it allows one to study the front propagation of the solutions to the equations \cite{Gar2011,FKT2016,FT2017c};
it also yields the non-existence of other stationary solutions between the given two (see \cite[Proposition 5.12]{FKT2015} 
and cf. a discussion in \cite{DD2003}) and allows to demonstrate instability of non-monotonic traveling waves, cf.~\cite{Hag1981}.

Since the hair-trigger effect means just that a level set for a solution to \eqref{eq:basic} is going to contain an arbitrary large compact in $\X$ when time grows, it is naturally based on a estimate from below for the solution. Note that,
for the class of equations of the form \eqref{eq:basic} with a non-negative operator $G$ (see the assumption \eqref{assum:Gpositive} below), one can estimate the corresponding non-negative solution from above by the solution to the linearization of \eqref{eq:basic} at zero.
 Indeed, by the Duhamel's principle, if $v(\cdot ,0) \equiv u(\cdot, 0)$  and $\partial_t v = \ka a*v - mv$, then 
 $u(\cdot,t)\leq v(\cdot,t)$ point-wise for all $t\geq0$.
Then one can use estimates on $v$ (see e.g. \cite{GKPZ2018, FT2017b, AMRT2010}) to estimate $u$.
However, an estimate from below appears much more delicate problem, that seems to be typical for monostable-type evolution equations, since the
nonlinear structure of the equation \eqref{eq:basic} is essential in this case.

Both nonlocal diffusion and, in general nonlocal, reaction in \eqref{eq:RDform} require new methods in the proof of the hair-trigger effect. The mentioned results for the local diffusion (the Laplace operator instead of $L$ in \eqref{eq:RDform}--\eqref{eq:nonlocdif}) were 
based on the application of an auxiliary boundary-value problem \cite{Kan1964} (which works for $d=1$ only) or, in addition to properties of the Laplace operator, 
on the locality of the reaction term \cite{AW1975, AW1978}. 
These approaches are difficult (if possible at all) to repeat for \eqref{eq:RDform} even for a local reaction in $\X$. Stress also that, in the case of a nonlocal reaction in \eqref{eq:RDform}, the comparison principle (which is necessary for the hair-trigger effect) requires additional restrictions (see Theorem \ref{thm:compar_pr_basic} and also Remark~\ref{rem:neccond}).  

Our approach is based on an extension of the classical Weinberger's result for discrete dynamical systems \cite{Wei1982a} to the continuous-time dynamics defined by \eqref{eq:basic}. That result required, additionally, specific restrictions on the initial condition to \eqref{eq:basic} (see the beginning of Section~\ref{sec:hair_trigger} for  details) or, equivalently, it requires an additional analysis for small level sets of the solution to \eqref{eq:basic} (which we provide in Propositions~\ref{prop:subsolution}--\ref{prop:useBrandle} below).
A disadvantage of this approach is that we apply `a black box',
meaning that the result sacrifices the complete understanding for the 
behaviour of large level-sets of $u$.
On the other hand, our approach is rather general and could be applied to other 
(nonlocal) evolution equations.

The paper is organized as follows. We prove the hair-trigger effect for \eqref{eq:basic} (Theorems~\ref{thm:ht1myes}, \ref{thm:ht1mno}) in Section~\ref{sec:hair_trigger}, applying Weinberger's results \cite{Wei1982a} and getting its time-continuous counterpart for \eqref{eq:basic} in Proposition~\ref{prop:hairtrigger_general}; we also, in Propositions~\ref{prop:subsolution}--\ref{prop:useBrandle}, get rid of the restrictions on the initial conditions imposed in Weinberger's paper. The proof is done under additional assumptions on $G$ presented in Section~\ref{sec:assumptionsandresults}, which, in particular, ensure the comparison principle (Theorem~\ref{thm:compar_pr_basic}).  
In Sections~\ref{sec:exist_uniq} and \ref{sec:comparison_pr}, we prove the existence/uniqueness (Theorem~\ref{thm:exist_uniq_BUC}) and the comparison principle (Theorem~\ref{thm:compar_pr}) for some generalizations of \eqref{eq:basic}.

\section{Assumptions and main results}\label{sec:assumptionsandresults}

Recall, that we treat $u=u(x,t)$ as the local density of a system at the point $x\in\X$ and at the moment of time $t\in \R_+:=[0,\infty)$. We assume that the initial condition $u_0$ to \eqref{eq:basic} is a bounded function on $\X$.

Namely, we will consider the following Banach spaces of real-valued functions on $\X$: the space $\Cb$ of bounded continuous functions on $\X$ with $\sup$-norm, the space $\Buc$ of bounded uniformly continuous functions on $\X$ with $\sup$-norm, and the space $L^{\infty}(\X)$ of essentially bounded (with respect to the Lebesgue measure) functions on $\X$ with $\esssup$-norm.

Let $E$ be either of the spaces $\Buc$, $\Cb$ or $L^{\infty}(\X)$ with the corresponding norm denoting by $\|\cdot\|_E$. For an interval $I\subset\R_+$, let $C(I\to E)$ and $C^1(I\to E)$ denote the sets of all continuous and, respectively continuously differentiable, $E$-valued functions on $I$.

\begin{defn}\label{def:classicalsol}
Let $I$ be either a finite interval $[0,T]$, for some $T>0$, or the whole $\R_{+}:=[0,\infty)$. A function 
$u\in \Y_I:= C(I\to E)\cap C^1((I\setminus\{0\})\to E)$ which satisfies \eqref{eq:basic} and such that $u(\cdot,0)=u_0(\cdot)$ in $E$ is said to be a \emph{classical solution} to \eqref{eq:basic} on $I$. For brevity, we denote also
\begin{equation}\label{eq:defYT}
\Y_T:=\Y_{[0,T]}, \quad T>0; \qquad \Y_\infty:=\Y_{\R_+}.
\end{equation}
\end{defn}

We will write $v\leq w$, for $v,w\in E$, if $v(x) \leq w(x)$, $x\in\X$. Here and below, for the case $E=L^\infty(\X)$, we will treat the latter inclusion a.e. only. Set also, for an $r>0$, 
\[
	E_r^+:=\{v\in E: 0 \leq v \leq r\}.
\]

We denote by $T_y:E\to E$, $y\in\X$, the translation operator, given by
\begin{equation}\label{shiftoper}
	(T_y v)(x)=v(x-y), \quad x\in\X.
\end{equation}

A sequence of functions $(v_n)_{n\in\N}\subset E$ is said to be convergent  to a function $v \in E$ locally uniformly
if $(v_n)_{n\in\N}$ converges to $v$ uniformly on all compact subsets of $\X$. We denote this by
\[
	v_n \locun v,\quad n\to \infty,
\] 

Let also $B_r(x_0)$ denote the ball in $\X$ with the radius $r>0$ centered at the $x_0\in\X$. In the case $x_0=0\in\X$, we will just write $B_r:=B_r(0)$.

In Section~\ref{sec:exist_uniq}, we prove an existence and uniqueness result for a more general equation than \eqref{eq:basic}; it can be read in the case of \eqref{eq:basic} as follows
\begin{thm}\label{thm:exist_uniq_basic}
	Let $0\leq a\in L^1(\X)$ and \eqref{eq:normalized} hold. Let $G:E\to E$ be such that $Gv\geq0$ for all $0\leq v\in E$ and, for some $\kappa>0$,
\[
	\|Gv-Gw\|_E \leq e^{\kappa r} \|v-w\|_E, \quad v,w\in E^+_r, \ r>0.
\]
Then, for any $T>0$ and $0\leq u_0\in E$, there exists a unique nonnegative classical solution $u$ to \eqref{eq:basic} on $[0,T]$. In particular, $u\in\Y_\infty$.
\end{thm}

To exclude the trivial case when $\|u(\cdot,t)\|_E$ converges to $0$ uniformly in time, we assume that
\begin{assum}\label{assum:kappa>m}
	\beta:=\ka-m>0.
\end{assum}

We suppose that there exist two constant solutions $u\equiv 0$ and $u\equiv \theta>0$ to \eqref{eq:basic}, more precisely,
\begin{assum}\label{assum:Gpositive}
	\begin{gathered}
	\textit{there exists $\theta>0$ such that }\\
		0=G0 \leq Gv\leq G\theta=\beta,\quad v \in E_\theta^+.
	\end{gathered} 
\end{assum}

We will also assume that $G$ is (locally) Lipschitz continuous in $E_\theta^+$, namely, 
\begin{assum}\label{assum:Glipschitz}
	\begin{gathered}
		\textit{there exists $\lt>0$, such that}\\
		\|Gv-Gw\|_E \leq \lt \|v-w\|_E,\quad v,w \in E_\theta^+.
	\end{gathered}
\end{assum}

We restrict ourselves to the case when the comparison principle for \eqref{eq:basic} holds.
Namely, we assume that the right-hand side of \eqref{eq:basic} is a (quasi-)monotone operator:
\begin{assum}\label{assum:sufficient_for_comparison}
	\begin{gathered}
		\textit{for some $p\geq0$ and for any $v,w\in E_\theta^+$ with $v\leq w$},\\
  	\ka a*v -v\, Gv + pv \leq \ka a*w -w\, Gw + pw.
	\end{gathered}
\end{assum}

In Section~\ref{sec:comparison_pr}, we also prove that the comparison principle holds for a more general equation than \eqref{eq:basic}; in the case of \eqref{eq:basic} it gives the following result.
\begin{thm}\label{thm:compar_pr_basic}
	Let \eqref{assum:kappa>m}--\eqref{assum:sufficient_for_comparison} hold. 
	\begin{enumerate}
		\item Let $T>0$ be fixed and $u_1,u_2\in\Y_T$ be such that, for all $t\in(0,T]$, $x\in\X$,
    \begin{gather*}
		\frac{\partial u_1}{\partial t} - \ka a*u_1 +mu_1 +u_1Gu_1 \leq \frac{\partial u_2}{\partial t} - \ka a*u_2 +mu_2 +u_2Gu_2,\\
		0 \leq u_1(x,t)\leq\theta, \qquad 0 \leq u_2(x,t)\leq \theta,\\
      0\leq u_{1}(x,0)\leq u_{2}(x,0)\leq \theta.
    \end{gather*}
    Then, for all $t\in[0,T]$, $x\in\X$,
	\begin{equation}\label{eq:comparineq}
		0\leq u_{1}(x,t)\leq u_{2}(x,t) \leq \theta.
	\end{equation}
\item Let $u\in\Y_\infty$ be a classical solution to \eqref{eq:basic}, given by Theorem~\ref{thm:exist_uniq_basic}, such that $0\leq u_0\leq \theta$. Then, for all $t\in\R_+$, $x\in\X$, 
	\[
		0\leq u(x,t)\leq \theta. 
	\]
In particular, combining two previous parts, we get the following statement.
\item Let functions $u_1,u_2\in \Y_\infty$  solve \eqref{eq:basic} and $0\leq u_{1}(x,0)\leq u_{2}(x,0)\leq \theta$, $x\in\X$. Then \eqref{eq:comparineq} holds for all $t\in\R_+$, $x\in\X$.
\end{enumerate}
\end{thm}

We assume next that the kernel $a$ is not degenerate at the origin, namely,
\begin{assum}\label{assum:a_nodeg}
	\textit{there exists $\varrho>0$ such that} \ a(x)\geq\varrho \text{ for a.a. } x\in B_\varrho(0).
\end{assum}

Stability of the solution to \eqref{eq:basic} with respect to the initial condition in the topology of locally uniform convergence requires continuity of $G$ in this topology:
\begin{assum}\label{assum:G_locally_continuous}
	\begin{gathered}	
		\textit{for any $v_n, v \in E_\theta^+$, such that $v_n \locun v$, $n \to \infty$, one has}\\
		Gv_n \locun Gv, \ n \to \infty.
	\end{gathered}
\end{assum}

We will consider the translation invariant case only:
\begin{assum}\label{assum:G_commute_T}
	\begin{gathered}
		\textit{let $T_y$, $y\in\X$, be a translation operator, given by \eqref{shiftoper}, then}\\
		(T_y G v)(x) = (G T_y v)(x),\quad v \in E_\theta^+,\  x\in\X.
	\end{gathered}
\end{assum}
Under \eqref{assum:G_commute_T}, for any $r\equiv const\in (0,\theta)$, $Gr\equiv const$. In this case, we assume also that
\begin{assum}\label{assum:G_increas_on_const}
  Gr < \beta, \qquad r\in(0,\theta).	
\end{assum}

In Section~\ref{sec:hair_trigger}, we prove the hair-trigger effect for the solutions to \eqref{eq:basic}. For technical reasons, it will be done separately  for kernels with and without the first moment.
Namely, for the kernels which satisfy the condition
\begin{assum}\label{assum:first_moment_finite}
	\int_\X \lvert y\rvert a(y)dy<\infty,
\end{assum}
we set
\begin{equation}\label{firstfullmoment}
  \m:=\ka\int_\X x a(x)\,dx\in\X,
\end{equation}
and assume, additionally to \eqref{assum:sufficient_for_comparison}, that
\begin{assum}\label{assum:improved_sufficient_for_comparison}
	\begin{gathered}
		\textit{there exist $q\geq 0$, $\delta>0$, $0\leq b\in C^\infty(\X)\cap L^\infty(\X)$, such that}\\ 
		a(x)-b(x)\geq \delta\1_{B_\delta(0)}(x), \quad x\in\X,\\
  	w\, Gw\leq \ka b*w + qw, \quad w\in E_\theta^+.
	\end{gathered}
\end{assum}

\begin{rem}\label{rem:notequiv} We are going to formulate now our main results about the hair-trigger effect for a solution to \eqref{eq:basic}. It requires that the initial condition to \eqref{eq:basic} is not degenerate: if $E$ is a space of continuous functions, this means that $u_0$ is not identically equal to zero, $u_0\not\equiv0$. 
For a brevity of notations, in the case $E=L^\infty(\X)$, we will treat $u_0\not\equiv0$ as follows: there exists $\delta>0$ and $x_0\in\X$, such that $u_0(x)\geq \delta$ for a.a.~$x\in B_\delta(x_0)$. 
\end{rem}

Then we can formulate the following 
\begin{thm}\label{thm:ht1myes}
	Let the conditions \eqref{assum:kappa>m}--\eqref{assum:improved_sufficient_for_comparison} hold.
	Let $u_0\in E_\theta^+$, $u_0\not\equiv0$ (cf. Remark \ref{rem:notequiv}), and let $u$ be the corresponding solution to \eqref{eq:basic}.
Then, for $\m$ defined by \eqref{firstfullmoment} and any compact set $K\subset \X$,
  \begin{equation}\label{eq:htformula}
    \lim_{t\to\infty} \essinf_{x\in K} u(x+t\m,t)=\theta.
  \end{equation}
\end{thm}
	\begin{rem}
		Note that the correction term $t\m=t\ka\int_\X ya(y)dy$ in \eqref{eq:htformula} equals to the expected value of the compound Poisson process with the probability density $a$ and the intensity $\ka$.
	\end{rem}

An evident example of a probability kernel with an infinite first moment is the density  
$a(x) = c(1+\lvert x\rvert ^2)^{-\frac{1+d}{2}}$, $x\in\X$ of the multivariate Cauchy distribution; here $\lvert \cdot\rvert $ denotes the Euclidean norm in $\X$, and $c$ is the normalizing factor to ensure \eqref{eq:normalized}.
To include this and other cases, for the kernels which do not satisfy \eqref{assum:first_moment_finite}, we consider the following assumption:
\begin{assum}\label{assum:approx_of_basic}
	\begin{gathered}
		\textit{for each $n\in\N$, let there exist} 
		\\
		0\leq a_n\in L^1(\X), \quad \ka_n>0, \quad G_n:E\to E, \quad \theta_n\in(0,\theta]\\ 
		\textit{which satisfy \eqref{assum:kappa>m}--\eqref{assum:improved_sufficient_for_comparison} instead of $a$, $\ka$, $G$, $\theta$,}\\
		\textit{ correspondingly, such that}\\
		\m_n:=\ka_n\int_{\X}xa_n(x)dx\in\R, \quad \theta_n\geq\theta-\frac{1}{n}, \ n\in \N,\\ 
		\ka_n a_n*w -wG_nw \leq \ka a*w - wGw,\quad w\in E_{\theta_n}^+.
	\end{gathered}
\end{assum}
Then the following counterpart of Theorem~\ref{thm:ht1myes} holds.
\begin{thm}\label{thm:ht1mno}
	Let the condition \eqref{assum:approx_of_basic} hold.
Let $u_0\in E_\theta^+$, $u_0\not\equiv0$ (cf. Remark \ref{rem:notequiv}), and let $u$ be the corresponding solution to \eqref{eq:basic}.
Then, for any compact set $K\subset \X$ and for any $n\in\N$,
	\begin{equation*}
		\theta-\frac{1}{n} \leq \liminf_{t\to\infty} \essinf_{x\in K} u(x+t\m_n,t) \leq	\limsup_{t\to\infty} \essinf_{x\in K} u(x+t\m_n,t) \leq \theta.
	\end{equation*}
	In particular, if $\m_n=\widetilde{\m}\in\R$  for  all $n\geq n_0\in\N$, then
  \begin{equation*}
    \lim_{t\to\infty} \essinf_{x\in K} u(x+t\widetilde{\m},t)=\theta.
  \end{equation*}
\end{thm}

In particular, if \eqref{assum:kappa>m}--\eqref{assum:improved_sufficient_for_comparison} hold and $\m=0\in\X$ or if \eqref{assum:approx_of_basic} holds and $\m_n=0\in\X$ for all $n\geq n_0\in\N$, then one gets the desired hair-trigger effect described above. 

\begin{rem}
Note that, indeed, for a properly `slanted' anisotropic kernel $a$ with $\m\neq 0\in\X$, the solution to \eqref{eq:basic} may converge to $0$ uniformly on any ball centered at the origin, whereas it will converge to $\theta$ on the `time-moving' ball according to Theorems~\ref{thm:ht1myes} or~\ref{thm:ht1mno}; see \cite{FKT2015} for the corresponding result in the case of the Example \ref{ex:Gu_equals_aminus_conv_u} described below.
\end{rem}

\subsection*{Examples}
\begin{exmp}[\bf Reaction--diffusion equation with a local reaction]\label{ex:local_reaction}
 A particular example of \eqref{eq:RDform}, with $F(u) = f(u)$ for a function $f:\R\to\R$, was considered e.g. in \cite{BCV2016,BCGR2014,CDM2013,CDM2008,CD2007,Cov2007,Yag2009,Gar2011,AGT2012,SLW2011,AC2016}. We assume \eqref{assum:kappa>m} and \eqref{assum:a_nodeg} as before, whereas the assumptions \eqref{assum:Gpositive}--\eqref{assum:sufficient_for_comparison}, \eqref{assum:G_locally_continuous}--\eqref{assum:G_increas_on_const}, \eqref{assum:improved_sufficient_for_comparison} are fulfilled if only
\begin{gather*}
	f \textit{ is Lipschitz continuous on } [0,\theta]; \\
	\lim\limits_{r\to 0 +} \frac{f(r)}{r} = \beta;\\ 
	f(0)=f(\theta)=0;\quad 0<f(r) \leq \beta r,  \ r\in(0,\theta). 
\end{gather*}

If \eqref{assum:first_moment_finite} does not hold, then, to fulfill \eqref{assum:approx_of_basic}, it is enough to take  $\ka_n = \ka$, $a_n(x):= \1_{\La_n}(x) a(x)$, provided that $\La_n\subset \X$ are such that $\La_n\uparrow \X$ and $\int_{\La_n} xa(x) dx = \widetilde{\m}$. In particular, if $a(-x)=a(x)$, $x\in\X$, one can take $\La_n:= B_n(0)$.
\end{exmp}

\begin{exmp}[\bf Spatial logistic equation: $\boldsymbol{Gu=\ka^- a^-*u}$]\label{ex:Gu_equals_aminus_conv_u}
Let $\kam>0$ and $a^-(x)$ be a probability kernel.
We consider $Gu=\kam a^-*u$, i.e. \eqref{eq:basic} has the form
\[
  \frac{\partial u}{\partial t} = \ka (a*u) - \kam u(a^-*u) - mu.
\]
This equation first appeared, for the case $\ka a=\kam a^-$, $m=0$, in \cite{Mol1972a,Mol1972}; for the case $\ka a=\kam a^-$, $m>0$ in  \cite{Dur1988}, and for the different kernels in \cite{BP1997}, where the so-called Bolker--Pacala model of spatial ecology was considered.
The equation was rigorously derived from the Bolker--Pacala model in \cite{FM2004} for integrable $u$ and in \cite{FKK2011a} for bounded $u$.
The long-time behavior of this equation was studied in \cite{FKT2015,FKT2016,FT2017c}, see also \cite{PS2005}.

We assume \eqref{assum:kappa>m} and \eqref{assum:a_nodeg} as before. Under  \eqref{assum:kappa>m}, we have in this case $\theta=\dfrac{\ka-m}{\kam}>0$. Then the conditions \eqref{assum:Gpositive}--\eqref{assum:Glipschitz}, \eqref{assum:G_locally_continuous}--\eqref{assum:G_increas_on_const} are satisfied. The condition \eqref{assum:sufficient_for_comparison} holds if and only if 
\begin{equation}\label{eq:compofkernels}
\ka a(x) \geq (\ka-m) a^-(x),\qquad x\in\X.
\end{equation}
Condition \eqref{assum:improved_sufficient_for_comparison} holds if we additionally assume that there exists  $\delta>0$, such that
\[
	\ka a(x) - (\ka-m) a^-(x) \geq \delta\1_{B_\delta(0)}(x), \qquad x\in\X.
\]
In this case we can put, in \eqref{assum:improved_sufficient_for_comparison}, $b(x) = (\ka-m) a^-(x)$, $q=0$.

If \eqref{assum:first_moment_finite} does not hold, then, to fulfill \eqref{assum:approx_of_basic}, one can proceed as in the previous example. Namely, we define $a_n$ as before, and we set $G_nu=\kam a^-_n*u$, where $a_n^-(x):=\1_{\La_n}(x)a^-(x)$, $x\in\X$.
\end{exmp}

\begin{exmp}[\bf The case $\boldsymbol{Gu=\kam a^-*u- g_1(a^-*u)}$]\label{ex:nonlocal_general}
Let $g(s) = \kam s - g_1(s)$, where $\kam>0$, $g_1:[0,\theta]\to \R_+$ is increasing  and Lipschitz continuous, such that $g_1(s) = o(s)$, as $s\to 0$ and $\kam s \geq g_1(s)$, for $s\in (0,\theta)$. We define $Gv=g(a^-*v)$, where $a^-$ is a probability kernel. Namely, we consider the following equation,
    \[
    	\frac{\partial u}{\partial t} = \ka (a*u) - \kam u(a^-*u) + u g_1(a^-*u) - mu.
    \]
		As in the previous example, \eqref{assum:sufficient_for_comparison} holds if and only if \eqref{eq:compofkernels} holds.
		The rest of the assumptions can be characterized straightforward.
    Typical example is $g(s) = \beta \bigl(1-\bigl(1-\frac{s}{\theta}\bigr)^n\bigr)$. In this case, the corresponding reaction  is 
    \[
    F(u) = \frac{\beta}{\theta^n} u (\theta-a^-*u)^n.
    \]
	\end{exmp}    

\section{Existence and uniqueness}\label{sec:exist_uniq}
In this Section, we will show the existence and uniqueness of non-negative solutions to a generalized version of \eqref{eq:basic} on $\R_+$, see Theorem~\ref{thm:exist_uniq_BUC} below. 
Note that the equation \eqref{eq:basic} itself is a semi-linear evolution (parabolic) equation on $E$. The condition \eqref{assum:Glipschitz} ensures that the nonlinear term $u\,Gu$ in \eqref{eq:basic} is locally Lipschitz. 
The general theory of semi-linear parabolic equations (see e.g. \cite[Theorem 6.1.4]{Paz1983}) provides existence and uniqueness of the so-called \emph{mild}
solution to \eqref{eq:basic} on the time interval $[0,t_{\mathrm{max}})$ for some $t_{{\mathrm{max}}}\leq\infty$. 
Since the operator \eqref{eq:nonlocdif} in \eqref{eq:basic} is bounded on $E$ and $G$ is continuous, this solution will be the classical one.
Moreover, if $t_{\mathrm{max}}<\infty$, then, with necessity, $\|u(\cdot,t)\|_{E}\to\infty$, as $t\nearrow t_{\mathrm{max}}$.
However, given $u_0\geq0$, the general theory does not ensure that $u(\cdot,t)\geq0$, $t\in[0,t_{\mathrm{max}})$.

\begin{rem} 1) Note that if we know \emph{a priori} that $u$ is non-negative on $[0,t_{\mathrm{max}})$, then $t_{\mathrm{max}}=\infty$, 
provided that $Gv\geq0$ for all $0\leq v\in E$ (cf. \eqref{assum:Gpositive} and the conditions of Theorem~\ref{thm:exist_uniq_basic}). Indeed, Duhamel's principle would imply then that $0\leq u(x,t)\leq e^{-mt}e^{tA}u_0(x)$, where $(Av)(x):=(a*v)(x)$, and hence  $\|u(\cdot,t)\|_{E}$ remains bounded on any finite time interval.

2) Another sufficient condition that would guarantee $t_{\mathrm{max}}=\infty$ is, therefore, the \emph{a priori} global boundedness of $u$. In the case of the `local' operator $G$, corresponding to the local reaction $Fu=f(u)$ in \eqref{eq:RDform} (cf. Example \ref{ex:local_reaction}), the global boundedness will follow from the comparison arguments considered in the Section~\ref{sec:comparison_pr} below (cf.~Theorem~\ref{thm:compar_pr_basic}). However, the case of a nonlocal operator $G$, and hence a nonlocal reaction $F$, would require a restrictive assumption \eqref{assum:sufficient_for_comparison} for comparison. Moreover, one can modify the example in \cite[pp.  2738--2739]{HR2014} to show that, in general, a solution to \eqref{eq:basic} does not need to be globally bounded on $\R_+$.

3) Note also, that any globally Lipschitz reaction $F$ (and hence globally Lipschitz product $u\,Gu$) would lead to $t_{\mathrm{max}}=\infty$ (see e.g. \cite[Theorem 3.2, 3.3]{Die1978}, \cite[Theorem 2.1]{Die1978a}).
\end{rem}

To avoid aforementioned additional assumptions for the non-local case of $G$ and $F$, we consider here a direct proof of the existence and uniqueness of non-negative solutions to (a generalized version of) the equation \eqref{eq:basic}. Our proof uses standard fixed point-arguments to get existence and uniqueness on consecutive time intervals $[\Upsilon_j,\Upsilon_{j+1}]$, $j\geq0$, $\Upsilon_0=0$. Then, using Lemma~\ref{lem:recurrence_sequence} below, we will show that $\sum_{j\geq0} (\Upsilon_{j+1}-\Upsilon_j)=\infty$ that implies the existence and uniqueness on an arbitrary time-interval.

\begin{lem}\label{lem:recurrence_sequence}
	Let $\{r_n\} _{n\in\N}$ be a sequence of numbers, such that $r_1>0$ and the following recurrence relation holds
	\begin{equation}\label{eq:recurrence_relation}
		r_{n+1} = r_n + p e^{-q r_n}, \quad n\in\N,
	\end{equation}
	where $p,q>0$. Then the series $\sum\limits_{n\in\N} \dfrac{1}{r_n e^{q r_n}}$ is divergent.
\end{lem}

\begin{proof} 
By \eqref{eq:recurrence_relation}, $r_n$, $n\in\N$ is a positive increasing sequence. Passing to the limit in \eqref{eq:recurrence_relation} when $n \to \infty$, one gets that $r_n \to \infty$, as $n \to \infty$. Hence, without loss of generality, one can assume that $b_n:=e^{-q r_n}< (pq)^{-1}$, $n\in\N$. One can rewrite then \eqref{eq:recurrence_relation} as follows: $b_{n+1}=b_n e^{-pq b_n}$. 
It is straightforward to check that
	\begin{equation*}
	  \frac{x}{1+pqx(e-1)}\leq y e^{-pqy} ,\quad 0 < x \leq y \leq \frac{1}{pq},
	\end{equation*}
Therefore, if we set $c_1:=b_1$ and $c_{n+1} := \frac{c_n}{1+pq(e-1)c_n}$, $n\in\N$, we get $c_n \leq b_n,\ n\in\N$. On the other hand, $\frac{1}{c_{n+1}} =\frac{1}{c_n} + pq(e-1)$, that leads to 
\begin{equation}\label{eq:cnasymp}
\frac{1}{c_{n+1}} = \frac{1}{c_1} + n(e-1)pq, \quad n\in\N.
\end{equation}
Therefore,
	\begin{equation*}
		\sum_{n\in\N} \frac{1}{r_n e^{q r_n}} = \sum_{n\in\N} \frac{b_n}{-q \ln b_n} \geq \sum_{n\in\N} \frac{c_n}{-q \ln c_n} = \infty,
	\end{equation*}
since, by \eqref{eq:cnasymp},
\[
	\frac{c_n}{-\ln c_n}\sim \frac{1}{pq(e-1)n \ln n}, \quad n\to\infty.
\]
	The statement is proved.
\end{proof}

Let $I\subset\R_+$ be a closed interval. The set $C_{b}(I\rightarrow E)$ of all continuous bounded $E$-valued functions on $I$ becomes a Banach space being equipped with the norm
\begin{equation*}
	\Vert u\Vert_{C_{b}(I\rightarrow E)}:=\sup\limits _{t\in I}\Vert u(\cdot,t)\Vert_{E}.
\end{equation*}
For simplicity of notation, we denote also
\begin{equation}
\label{eq:norm_notation} 
\begin{aligned}
	\|u\|_{T_1,T_2}&:=\|u\|_{C_b([T_1,T_2]\to E)}, && 0< T_1<T_2;\\[2mm]
	\|u\|_{T}&:=\|u\|_{C_b([0,T]\to E)},  && T>0.
\end{aligned}
\end{equation}
We are ready to prove now the existence and uniqueness result. 

\begin{thm}\label{thm:exist_uniq_BUC}
	Let $A ,G:E\to E$ be such that $Gv\geq0$ and $Av\geq0$ for all $0\leq v\in E$, and, for some $\kappa, \ka>0$,
\begin{align}
	\|A  v-A  w\|_E &\leq \ka \|v-w\|_E, && v,w \in E,\ v\geq0,\ w\geq0, \label{eq:F_loc_lip}\\
	\|Gv-Gw\|_E &\leq e^{\kappa r} \|v-w\|_E, && v,w\in E^+_r,\ r>0. \label{eq:G_loc_lip}
\end{align}
Then, for any $T>0$ and $0\leq u_0\in E$, there exists a unique nonnegative classical solution $u\in \Y_T$ (cf.~Definition~\ref{def:classicalsol}) to the equation 
\begin{equation}
	\begin{cases}
		  \dfrac{\partial u}{\partial t}(x,t) = (A  u)(x,t)-mu(x,t)-u(x,t)(Gu)(x,t),\\[2mm]
		  u(x,0) = u_0(x),
	\end{cases}\label{eq:basic_gen}
\end{equation}
where $t\in (0,T]$, $x\in\X$.
\end{thm}
\begin{proof}
First, we note that, by \eqref{eq:F_loc_lip},
\begin{equation}\label{eq:F_norm_est}
	\|A  v\|_E \leq \|A  0 \|_E+\ka\|v\|_E, \quad 0\leq v\in E.
\end{equation}
We set $f_0:= \|A  0\|_E$. 

Let $T>0$ be arbitrary. Take any $0\leq v\in C_b([0,T]\to E)$. For any $\tau\in[0,T)$, consider the following linear equation in the space $E$ on the interval $[\tau,T]$:
	\begin{equation}
		\begin{cases}
				\dfrac{\partial u}{\partial t}(x,t) = (A  v)(x,t) -mu(x,t) -u(x,t)(Gv)(x,t), &t\in (\tau,T]\\[2mm]
				u(x,\tau) = u_\tau(x), 
		\end{cases}\label{eq:exist_uniq_BUC:basic_lin}
	\end{equation}
	where $0\leq u_\tau\in E$, $\tau>0$, and $u_0$ is the same as in \eqref{eq:basic_gen}.
By assumptions on $A $ and $G$, we have that $A  v, Gv \in C_b([0,T]\to E)$ for all ${v\in C_b([0,T]\to E)}$.
In the right-hand side of \eqref{eq:exist_uniq_BUC:basic_lin}, there is a time-dependent linear bounded operator (acting in $u$) in the space $E$ whose coefficients are continuous on $[\tau,T]$.
Therefore, there exists a unique solution to \eqref{eq:exist_uniq_BUC:basic_lin} in $E$ on $[\tau,T]$, given by $u=\Phi_\tau v$ with
	\begin{equation}
		(\Phi_\tau v)(x,t):=(Bv)(x,\tau,t)u_\tau(x)+\int_\tau^t(Bv)(x,s,t)(A v)(x,s)\,ds,
		\label{eq:exist_uniq_BUC:Phi_v}
	\end{equation}
	for $x\in\X$, $t\in[\tau,T]$, where we set
	\begin{equation}
		(Bv)(x,s,t):=\exp\biggl(-\int _{s}^t \big(m + (Gv)(x,p) \big)dp\biggr),\label{eq:exist_uniq_BUC:B}
	\end{equation}
	for $x\in\X$, $t,s\in[\tau,T]$. Note that, in particular, $(\Phi_\tau v)(\cdot,t), (Bv)(\cdot,s,t)\in E$.	Clearly, $(\Phi_\tau v)(x,t)\geq0$ and, for any $\Upsilon\in(\tau,T]$,
	\begin{equation}\label{eq:est}
		\|\Phi_\tau v(\cdot,t)\|_E \leq \|u_\tau\|_E +(f_0+\ka  \|v\|_{\tau,\Upsilon})(\Upsilon-\tau), \qquad t\in[\tau,\Upsilon],
	\end{equation}
	where we used \eqref{eq:F_norm_est} and the notation \eqref{eq:norm_notation}. Therefore, $\Phi_\tau$ maps $\{ 0\leq v\in C_b([\tau,\Upsilon]\to E)\}$ into itself, $\Upsilon\in(\tau,T]$.

	For any $T_2>T_1\geq0$ and $r>0$, we define
	\begin{equation}\label{eq:defofballinxtplus}
		\x_{T_1,T_2}^+(r):=\bigl\{ v\in C_b([T_1,T_2]\to E) \bigm| v\geq0, \|v\|_{T_1,T_2} \leq r\bigr\}.
	\end{equation}

	Let now $0\leq\tau < \Upsilon\leq T$, and take any $v,w\in\x_{\tau,\Upsilon}^+(r)$.
	By \eqref{eq:exist_uniq_BUC:Phi_v}, one has, for any $x\in\X$, $t\in[\tau,\Upsilon]$,
	\begin{equation}\label{estbyJ}
		\bigl|(\Phi_\tau v)(x,t)-(\Phi_\tau w)(x,t)\bigr|\leq J_1+J_2,
	\end{equation}
	where
	\begin{align*}
		J_1&:=\bigl|(Bv)(x,\tau,t)-(Bw)(x,\tau,t)\bigr|u_\tau(x),\\
		J_2&:=\int_\tau^t\bigl|(Bv)(x,s,t)(A  v)(x,s) -(Bw)(x,s,t)(A  w)(x,s)\bigr|\,ds.
	\end{align*}
	Clearly, for each $a\in L^1(\X)$, $f\in E$,
	\begin{equation}\label{convbdd}
    \bigl\lvert (a*f)(x)\bigr\rvert\leq \|f\|_E \,\|a\|_{L^1(\X)}.
  \end{equation}
	Since $|e^{-a}-e^{-b}|\leq |a-b|$, for any constants $a,b\geq0$, one has, by \eqref{eq:exist_uniq_BUC:B}, \eqref{convbdd},
	\begin{equation}
		J_1\leq e^{\kappa r}(\Upsilon-\tau)\lVert u_\tau\rVert_E \lVert v-w\rVert_{\tau,\Upsilon}.\label{eq:exist_uniq_BUC:Phi_est_i}
	\end{equation}
	Next, for any constants $a,b,p,q\geq0$,
	\begin{equation*}
		\bigl|pe^{-a}-qe^{-b}\bigr|\leq e^{-a}|p-q|+q\max\bigl\{e^{-a},e^{-b}\bigr\}|a-b|,
	\end{equation*}
	therefore, by \eqref{eq:exist_uniq_BUC:B}, \eqref{convbdd},
	\begin{align}
		J_2&\leq  \ka\int_\tau^t(Bv)(x,s,t)\,ds\lVert v-w\rVert_{\tau,\Upsilon}\notag\\
			&\quad+\int_\tau^t\max\bigl\{(Bv)(x,s,t),(Bw)(x,s,t)\bigr\}|(A  w)(x,s)|\notag (t-s)e^{\kappa r}\lVert v-w\rVert_{\tau,\Upsilon}\,ds\notag\\
			&\leq \ka(\Upsilon-\tau) \lVert v-w\rVert_{\tau,\Upsilon} + e^{\kappa r}(f_0+\ka\lVert w\rVert_{\tau,\Upsilon})\lVert v-w\rVert_{\tau,\Upsilon}\int_\tau^t e^{-m(t-s)}(t-s)\,ds\notag\\
			&\leq  \Bigl(\ka +(f_0+\ka\lVert w\rVert_{\tau,\Upsilon})\frac{e^{\kappa r}}{me}\Bigr) (\Upsilon-\tau) \lVert v-w\rVert_{\tau,\Upsilon},\label{est2in1}
	\end{align}
	as $re^{-r}\leq e^{-1}$, $r\geq0$.

		Take any $\mu\geq\lVert u_\tau\rVert_E$. By \eqref{eq:est}--\eqref{est2in1}, one has,
	\begin{align*}
		\bigl|(\Phi_\tau v)(x,t)-(\Phi_\tau w)(x,t)\bigr| &\leq
		\Bigl(\mu e^{\kappa r} +\ka +(f_0+\ka r) \frac{e^{\kappa r}}{me} \Bigr) (\Upsilon -\tau)\lVert v-w\rVert_{\tau,\Upsilon},\\
		\bigl|(\Phi_\tau v)(x,t)\bigr| &\leq \mu+(f_0+\ka r) (\Upsilon -\tau).
	\end{align*}
	Therefore, $\Phi_\tau$ will be a contraction mapping on the set $\x_{\tau,\Upsilon}^+(r)$ if only
	\begin{equation*}
 		\Bigl(\mu e^{\kappa r}+\ka +(f_0+\ka r) \frac{e^{\kappa r}}{me}\Bigr) (\Upsilon -\tau) <1 \quad \text{and} \quad \mu+(f_0 +\ka r) (\Upsilon -\tau)\leq r.
	\end{equation*}
	If $\frac{f_0}{\ka} \leq r$, it is sufficient to show
	\begin{equation}\label{need1}
 		\Bigl(\mu e^{\kappa r} +\ka + 2\ka r \frac{e^{\kappa r}}{me}\Bigr) (\Upsilon -\tau) <1 \quad \text{and} \quad \mu + 2\ka r (\Upsilon -\tau)\leq r.
	\end{equation}

	Take for $\alpha\in(0,1)$,
	\begin{equation}\label{settings}
		\begin{gathered}
			r:=\mu+\alpha me^{1-\kappa \mu}, \qquad  \Upsilon :=\tau+\frac{\alpha me}{2\ka re^{\kappa r}}.
		\end{gathered}
	\end{equation}
	Then, the second inequality in \eqref{need1} holds, since $e^{\kappa r}$ is increasing, namely,
	\begin{equation*}
		\mu+2\ka r(\Upsilon-\tau) = \mu + \alpha me^{1-\kappa r} \leq \mu + \alpha me^{1-\kappa \mu} = r.
	\end{equation*}
	Next	
	\begin{equation*}
		\Bigl(\mu e^{\kappa r} + \ka + \frac{2 \ka r e^{\kappa r}}{me}\Bigr)(\Upsilon-\tau) = \frac{\alpha me \mu}{2\ka r} + \frac{\alpha me}{2re^{\kappa r}} +\alpha \leq \frac{\alpha me}{2\ka} + \frac{\alpha me} {2re^{\kappa r}} + \alpha.
	\end{equation*}
	In order to satisfy the second inequality in \eqref{need1} it is sufficient to check,
	\begin{equation*}
		\frac{\alpha me}{2\ka} + \frac{\alpha me} {2re^{\kappa \mu}} < 1-\alpha,
	\end{equation*}
	but $re^{\kappa \mu} = \mu e^{\kappa \mu} + \alpha me$, i.e. we need
	\begin{equation}\label{need111}
	  \frac{\alpha me}{2(\mu e^{\kappa \mu} + \alpha me)} + \frac{\alpha me} {2 \ka} < 1-\alpha.
	\end{equation}
	
		Choose $\alpha\in(0,1)$, such that $\frac{\alpha me}{2\ka} < 1-\alpha$, and then choose $\mu>0$ large enough to ensure \eqref{need111}. 
	As a result, one gets that $\Phi_\tau$ will be a contraction on the set $\x_{\tau,\Upsilon}^+(r)$ with $\Upsilon$ and $r$ given by \eqref{settings}; the latter set naturally forms a complete metric space. Therefore, there exists a unique $u\in \x_{\tau,\Upsilon}^+(r)$ such that $\Phi_\tau u=u$. This $u$ will be a solution to \eqref{eq:basic_gen} on $[\tau,\Upsilon]$.

To fulfill the proof of the statement, one can do the following.
Set $\tau:=0$, choose $r_0>\max \{\|u_0\|_E,\frac{f_0}{\ka}\} $ and $\alpha \in (0,1)$ that satisfy \eqref{need111} with $\mu=r_0$.
One gets a solution $u$ to \eqref{eq:basic_gen} on $[0,\Upsilon_1]$ with $\lVert u\rVert_{\Upsilon_1} \leq r_0+\alpha m e^{1-\kappa r_0} =: r_1$, $\Upsilon_1 = \frac{\alpha me^{1 -\kappa r_1}}{2\ka r_1}$.

	Iterating this scheme, take sequentially, for each $n\in\N$, $\tau:=\Upsilon_n$, $x\in\X$,
	\[
		r_{n}:=r_{n-1}+ \alpha me^{1-\kappa r_{n-1}}\geq \|u(\cdot, \Upsilon_n)\|_E.
	\]
	Since $r_{n}>r_{n-1}$ and $e^{\kappa r}$ is increasing, the same $\alpha$ as before will satisfy \eqref{need111} with $\mu=r_{n}$ as well. Then, one gets a solution $u$ to \eqref{eq:basic_gen} on $[\Upsilon_n,\Upsilon_{n+1}]$ with initial condition $u_{\Upsilon_n}$, where
	\begin{gather}\label{eq:Tn}
		\Upsilon_{n+1} := \Upsilon_n+\frac{\alpha me^{1-\kappa r_n}}{2\ka r_{n}},\\
	\shortintertext{and}
		\|u\|_{\Upsilon_n,\Upsilon_{n+1}} \leq r_{n}+ \alpha me^{1-\kappa r_n} = r_{n+1}.\notag
	\end{gather}
	As a result, we will have a solution $u$ to \eqref{eq:basic_gen} on intervals $[0,\Upsilon_1]$, $[\Upsilon_1,\Upsilon_2]$, \ldots, $[\Upsilon_n,\Upsilon_{n+1}]$, $n\in\N$.
By \eqref{eq:F_loc_lip}--\eqref{eq:G_loc_lip}, the right-hand side of \eqref{eq:basic_gen}, will be continuous on each of constructed time-intervals, therefore, one has that $u$ is continuously differentiable on $(0,\Upsilon_{n+1}]$ and solves \eqref{eq:basic} there. By \eqref{eq:Tn} and Lemma~\ref{lem:recurrence_sequence},
	\[
		\Upsilon_{n+1} = \frac{\alpha me}{2\ka} \sum_{j=0}^{n}  \frac{1}{r_j e^{\kappa r_j}} \to\infty , \quad n\to\infty,
	\]
	therefore, one has a solution to \eqref{eq:basic_gen} on any $[0,T]$, $T>0$.

To prove uniqueness, suppose that $v\in {C_b([0,T]\to E)}$ is a solution to \eqref{eq:basic_gen} on $[0,T]$, with $v(x,0)\equiv u_0(x)$, $x\in\X$. Choose $r_0 > \|v\|_T \geq \|u_0\|_E$. Since $\{r_n\}_{n\geq0}$ above is an increasing sequence, $v$ will belong to each of sets  $\x_{\Upsilon_n,\Upsilon_{n+1}}^+(r_{n+1})$, $n\geq0$, $\Upsilon_0:=0$, considered above. Then, being solution to \eqref{eq:basic_gen} on each $[\Upsilon_n,\Upsilon_{n+1}]$, $v$~will be a fixed point for $\Phi_{\Upsilon_n}$. By the uniqueness of such a point, $v$ coincides with $u$ on each $[\Upsilon_n,\Upsilon_{n+1}]$ and, thus, on the whole $[0,T]$.
	As a result, $u(x,t) = (\Phi_0u)(x,t)$, for $x\in\X$, $t\geq0$.	Since $u\in {C_b([0,T]\to E)}$, then $u = \Phi_0 u\in C^1((0,T]{\to}E)$. Thus $u$ is a classical solution to \eqref{eq:basic}. The proof is fulfilled.
\end{proof}

\begin{rem}
	Since $A  v:=\ka a*v$, $v\in E$, evidently satisfies conditions of Theorem \ref{thm:exist_uniq_BUC}, one gets Theorem~\ref{thm:exist_uniq_basic}.
\end{rem}

\begin{prop}\label{prop:loc_uni}
Let the conditions of Theorem~\ref{thm:exist_uniq_BUC} hold.
Suppose, additionally, that $A $ and $G$ are continuous on $\{0 \leq v\in E\}$ in the topology of locally uniform convergence, i.e. for any $v_n, v \in E$, $v_n\geq 0$, $v\geq 0$, with $v_n \locun v$, one has
\[
	A  v_n \locun A  v,\qquad  G v_n \locun G v, \qquad n \to \infty.
\]
Let $T>0$ be fixed and, for some $\varrho>0$, $\{u(\cdot,0), u_n(\cdot,0):n\in\N\}\subset E_\varrho^+$ be the initial conditions to \eqref{eq:basic_gen}, and let $\{u(\cdot,t), u_n(\cdot,t):n\in\N\}$ be the corresponding solutions to \eqref{eq:basic_gen} on $[0,T]$. Assume that $u_n(\cdot,0)\locun u(\cdot,0)$, $n\to\infty$. Then $u_n(\cdot,t)\locun u(\cdot,t)$, $n\to\infty$ uniformly in $t\in[0,T]$.
\end{prop}

\begin{proof}
By the proof of Theorem~\ref{thm:exist_uniq_BUC}, there exist $0=\tau_0<\tau_1<\ldots<\tau_N= T$ and $\varrho=r_0\leq r_1\leq\ldots\leq r_N=:r$, such that the following holds.
Let, for any 
$\tau=\tau_k$, $\Upsilon=\tau_{k+1}$, $0\leq k\leq N-1$, the mapping $\Phi_\tau$ be defined by \eqref{eq:exist_uniq_BUC:Phi_v} for $t\in[\tau,\Upsilon]$, with $u_\tau(x)=u(x,\tau)$, $x\in\X$; and, for each $n\in\N$, we set
	\[
			(\Phi_{\tau,n} v)(x,t):=(Bv)(x,\tau,t)u_{\tau,n}(x)+\int_\tau^t(Bv)(x,s,t)(A  v)(x,s)\,ds, 
	\]
	where $u_{\tau,n}(x)=u_n(x,\tau)$, $x\in\X$.
	Then $v\in \x_{\tau,\Upsilon}^+(r_{k+1})$, $\{u_\tau,u_{\tau,n}:n\in\N\}\subset E_{r_k}$  implies $\{\Phi_\tau v, \Phi_{\tau,n}v:n\in\N\}\subset \x_{\tau,\Upsilon}^+(r_{k+1})$, (cf. \eqref{eq:defofballinxtplus}).

	Prove that if, for some $\{w,w_n:n\in\N\}\subset \x_{\tau,\Upsilon}^+(r_{k+1})$, we have that $w_n(\cdot,t){\locun}w(\cdot,t)$, $n\to\infty$, uniformly in $t\in[\tau,\Upsilon]$, then 
\begin{equation}\label{eq:uniconv}
	\Phi_{\tau,n} w_n(\cdot,t)\locun \Phi_{\tau} w(\cdot,t), \quad n\to\infty,
\end{equation}
uniformly in $t\in[\tau,\Upsilon]$. Indeed, applying the inequalities,
\[
|e^{-a}-e^{-b}|\leq |a-b|, 
\qquad 
\bigl|pe^{-a}-qe^{-b}\bigr|\leq |p-q|+q|a-b|,
\]
for $a,b,p,q\geq0$, we get, for any bounded $\La\subset\X$,
\begin{align*}
	&\quad\ \1_\La(x)\bigl\lvert (\Phi_{\tau,n} w_n)(x,t) - (\Phi_{\tau} w)(x,t)\bigr\rvert \\
		&\leq  \1_\La(x)\bigl\lvert (\Phi_{\tau,n} w_n)(x,t) - (\Phi_{\tau,n} w)(x,t)\bigr\rvert + \1_\La(x)\bigl\lvert (\Phi_{\tau,n} w)(x,t) - (\Phi_{\tau} w)(x,t)\bigr\rvert \\
		&\leq \1_\La(x) \bigl\vert u_{\tau,n}(x) - u_\tau(x) \bigr\vert	+ r_k \int_\tau^t \1_\La(x)\bigl\lvert (Gw_n)(x,p)-(Gw)(x,p)\bigr\rvert\,dp \\
	&\quad+\int_\tau^t \1_\La(x)\bigl\lvert (A  w_n)(x,s)-(A  w)(x,s)\bigr\rvert\,ds
\\&\quad+\int_\tau^t \1_\La(x)\bigl\lvert (A  w)(x,s)\bigr\rvert\,\int_s^t \bigl\lvert (Gw_n)(x,p)-(Gw)(x,p)\bigr\rvert\,dp\, ds\\
	&\leq \bigl\lVert \1_\La\big(u_{\tau,n} - u_\tau\big) \bigr\rVert_{E} + r_k \int_\tau^\Upsilon \bigl\lVert \1_\La\big((Gw_n)(\cdot,p)-(Gw)(\cdot,p)\big) \bigr\rVert_{E}dp\\ 
	&\quad +\int_\tau^\Upsilon \bigl\lVert \1_\La \big((A  w_n)(\cdot,s)-(A  w)(\cdot,s)\big) \bigr\rVert_{E}  \,ds
\\&\quad+ (\|A  (0)\|_E+\ka r) \int_\tau^\Upsilon \int_s^\Upsilon  \bigl\lVert  \1_\La \big( (Gw_n)(\cdot,p)-(Gw)(\cdot,p)\big) \bigr\rVert_E\,dp\, ds.
\end{align*}
Hence \eqref{eq:uniconv} holds. 
Iterating this scheme, one gets that, for each $m\in\N$, $v\in \x_{\tau,\Upsilon}^+(r_{k+1})$, 
\begin{equation}\label{eq:convforeachm}
	(\Phi_{\tau,n})^m v(\cdot,t)\locun (\Phi_{\tau})^m v, \quad n\to\infty,
\end{equation}
uniformly in $t\in[\tau,\Upsilon]$. Therefore, for any bounded $\La\subset\X$,
\begin{align*}
&\quad\ \bigl\lvert \1_\La(x)(u_n(x,t)-u(x,t))\bigr\rvert\\
&\leq \bigl\lvert \1_\La(x)\bigl(u_n(x,t)-(\Phi_{\tau,n})^m v(x,t)\bigr)\bigr\rvert
+ \bigl\lvert \1_\La(x)\bigl((\Phi_{\tau,n})^m v(x,t)-(\Phi_{\tau})^m v(x,t)\bigr)\bigr\rvert
\\&\quad+ \bigl\lvert \1_\La(x)\bigl(u(x,t)-(\Phi_{\tau})^m v(x,t)\bigr)\bigr\rvert
\\&\leq \bigl\lVert u_n - (\Phi_{\tau,n})^m v\|_{\tau,\Upsilon}
+\sup_{t\in[\tau,\Upsilon]}\bigl\lVert \1_\La\bigl((\Phi_{\tau,n})^m v(\cdot,t)-(\Phi_{\tau})^m v(\cdot,t)\bigr)\bigr\rVert_{E}
\\&\quad+ \bigl\lVert u - (\Phi_{\tau})^m v\|_{\tau,\Upsilon},
\end{align*}
for any $m\in\N$. Passing $m$ to $\infty$, one gets then the statement by \eqref{eq:convforeachm}.
	\end{proof}

\section{Comparison principle}\label{sec:comparison_pr}
The comparison principle is a standard tool in studying parabolic- and elliptic-type equations, see e.g. \cite{GGIS1991,CIL1992}. For instance, it allows to estimate an unknown solution, constructing explicit sub- and super-solutions \cite{AW1975,Aro1977,AW1978}. See also \cite{Sch1980,Cov2007,CDM2008} for comparison results and its applications in studying traveling waves for non-local equations.   
	To the best of our knowledge, the first detailed proof of the comparison principle for the parabolic equation in the case of nonlocal diffusion \eqref{eq:nonlocdif} in \eqref{eq:RDform}, was done by Yagisita~\cite{Yag2009} in the case of globally Lipschitz KPP-type reaction $Fu=f(u)$  (see also \cite[Lemma D.1]{LPL2005}). 
	The comparison principle is often used in other articles without any reference on the proof.
	Also we do not know any result on the comparison principle in the case of a non-local reaction.

We will get in Theorem~\ref{thm:compar_pr} the comparison principle related to an abstract evolution equation 
\[
\dfrac{\partial u}{\partial t}(x,t)=(Hu)(x,t),
\]
where $H:E\to E$ is locally Lipschitz continuous and such that the operator $H+p$ is monotone on $E$ for some $p>0$. Here and below we use the same notation for a constant and for the operator of multiplication by this constant  in the space $E$.

\begin{rem}\label{rem:neccond}
For the equation \eqref{eq:basic}, the monotonicity of $H+p$ has the form \eqref{assum:sufficient_for_comparison}.
Note that in the case of a local operator $G$ (cf. Example \ref{ex:local_reaction}), there exists $p>0$, such that \eqref{assum:Glipschitz} implies \eqref{assum:sufficient_for_comparison}, and hence the comparison indeed does not require any additional assumptions. However, for a nonlocal $G$ the assumption \eqref{assum:sufficient_for_comparison} is restrictive. For instance, in Example~\ref{ex:Gu_equals_aminus_conv_u}, \eqref{assum:sufficient_for_comparison} is necessary and sufficient (and  hence optimal) condition to ensure the comparison principle in $E_\theta^+$, see~\cite[Remark~3.6]{FKT2015}.
\end{rem}	
	
We introduce some additional notations. For any $v\in E$, $r\in\R$, we set 
\[
(v\wedge r)(x) := \min\{v(x),r\}, \quad (v\vee r)(x) := \max\{v(x),r\}.
\]
Let $H:E\to E$. For any $u\in\Y_T$, cf. \eqref{eq:defYT}, and $r>0$, we define
\begin{equation}\label{Foper}
  (\mathcal{F}_r u)(x,t):=\dfrac{\partial u}{\partial t}(x,t)- H(0 \vee u\wedge r)(x,t),\qquad t\in(0,T],\ x\in\X.
\end{equation}
Here and below we consider the left derivative at $t=T$ only. 
\begin{thm}\label{thm:compar_pr}
	Let $H:E\to E$ and $h,p,r>0$ be such that $H$ is Lipschitz continuous on $E_r^+$ with the Lipschitz constant $h >0$, and $H+p$ is monotone on $E_r^+$, namely,
	\begin{align}
		\|H w-Hv\|_E &\leq h \|w-v\|_E,&& w,v\in E_r^+, \label{eq:H_lip}\\
  	Hv +pv &\leq Hw+pw,&&v\leq w,\ w,v\in E_r^+. \label{compare}
	\end{align}
  Let $T>0$ be fixed. Suppose that  $u_{1},u_{2}\in\Y_T$ are such that 
    \begin{align}
    	0\leq u_1(x,t),&\quad u_2(x,t) \leq r, && (x,t)\in\X \times(0,T], \label{eq:comparison_b}\\
     (\mathcal{F}_r u_{1})(x,t) &\leq(\mathcal{F}_r u_{2})(x,t), && (x,t)\in\X \times(0,T],\label{eq:max_pr_BUC:ineqFcal}\\
      u_{1}(x,0)&\leq u_{2}(x,0), && x\in\X. \label{eq:max_pr_BUC:ineq}
    \end{align}
  	Then $u_{1}(x,t)\leq u_{2}(x,t)$ for all $(x,t)\in\X \times[0,T]$. 
\end{thm}

\begin{proof}
Define, cf. \eqref{eq:max_pr_BUC:ineqFcal}, the following function
\begin{equation}
	\phi_r(x,t):=(\mathcal{F}_r u_{2})(x,t)-(\mathcal{F}_r u_{1})(x,t)\geq0, \label{dif_pos}
\end{equation}
for $(x,t)\in\X \times[0,T]$. For a constant $K>0$, which will be specified later, consider the mapping
\begin{align}
	\Theta(t,w):&=Kw  +e^{Kt}\big( H\big(0\vee(e^{-Kt}w+u_1)\wedge r \big) - H(u_1\wedge r)\big)\notag \\&\quad + e^{Kt} \phi_r(x,t), \qquad w\in C_b([0,T]\to E). \label{mapF}
\end{align}
We have, for $w \geq 0$,
\[
	0 \leq u_1\wedge r \leq (e^{-Kt}w+u_1)\wedge r \leq r.
\]
Since, for any $x\geq y\geq0$, $z\geq0$,
\[
	0 \leq x\wedge z - y\wedge z \leq x-y,
\]
one has, by \eqref{compare}, \eqref{dif_pos}, that  $0\leq w \in C_b([0,T]\to E)$ yields
\[
	\Theta(t,w) \geq (K-p)w + e^{Kt} \phi_r(x,t) \geq0, 
\]
if only $K\geq p$ that we will assume in the following.

Next, applying \eqref{eq:H_lip} to \eqref{mapF}, we will get that ${w \in C_b([0,T]\to E)}$ implies, for all $t\in[0,T]$,
\[
	\|\Theta(t,w)\|_T \leq (K+h )\|w\|_T + e^{Kt}\|\phi_r\|_T.
\]
Therefore, since $u_1,u_2\in\Y_T$ implies, by \eqref{Foper}, \eqref{dif_pos}, that $\phi_r\in C_b([0,T]\to E)$, one gets that $\Theta(t,w)\in C_b([0,T]\to E)$.

Define also the function
\[
v(x,t):=e^{Kt}(u_{2}(x,t)-u_{1}(x,t)), \quad x\in\X, \ t\in[0,T].
\]
Clearly, $v\in\Y_T$, and it is straightforward to check that
\[
 \Theta (t,v(x,t))=\frac{\partial}{\partial t} v(x,t), \quad x\in\X, \ t\in(0,T].
\]
Therefore, $v$ solves the following integral equation in $E$:
\begin{equation}
\begin{cases}
\displaystyle v(x,t)=v(x,0)+\int_0^t \Theta(s,v(x,s))ds, & \quad (x,t)\in\X {\times}(0,T],\\[3mm]
v(x,0)=u_{2}(x,0)-u_{1}(x,0), & \quad x\in\X,
\end{cases}\label{eq:max_pr_BUC:u2-u1_lin}
\end{equation}
where $v(x,0)\geq0$ by \eqref{eq:max_pr_BUC:ineq}.

Consider also another integral equation in $E$:
\begin{gather}
	\widetilde{v}(x,t)=(\Psi \widetilde{v})(x,t)\label{eq:max_pr_BUC:pos_sol} \\ 
\shortintertext{where, for $w\in C_b([0,T]\to E)$,}
	(\Psi w)(x,t):=v(x,0)+\int_0^t\max\bigl\{\Theta(s,w(x,s)),0\bigr\}\,ds.\label{eq:max_pr_BUC:Psi}
\end{gather}
If we take $\widetilde{T}<T$ such that the following inequality holds
\[
	q_1 := 2r(K+h ) + e^{KT} \|\phi_r\|_T \leq \frac{r}{\widetilde{T}},
\]
then, $w\in \mathcal{X}_{\widetilde{T}}^+(2r)$  yields $\Psi w \in \mathcal{X}_{\widetilde{T}}^+(2r)$, where, cf. \eqref{eq:defofballinxtplus}, $\mathcal{X}_{\widetilde{T}}^+(2r):=\mathcal{X}_{0,\widetilde{T}}^+(2r)$.
Let $w_{1},w_{2}\in\mathcal{X}_{\widetilde{T}}(2r)$; by \eqref{eq:H_lip}, \eqref{mapF}, we have, for all $(t,x)\in[0,\widetilde{T}]\times\X$,
\[
	|\Theta(t,w_1)-\Theta(t,w_2)| \leq (K + h ) \|w_1-w_2\|_{\widetilde{T}} =: q_2 \|w_1-w_2\|_{\widetilde{T}}.  
\]
Therefore, using the elementary inequality $\lvert \max\{a,0\}-\max\{b,0\}\rvert\leq |a-b|$, $a,b\in\R$, we obtain from \eqref{eq:max_pr_BUC:Psi}, that
\begin{equation*}
	\|\Psi w_{1}-\Psi w_{2}\|_{\widetilde{T}}\leq q_2 \widetilde{T} \|w_2-w_1\|_{\widetilde{T}}.
\end{equation*}
Therefore, for $\widetilde{T}< \max\{ \frac{r}{2q_1}, \frac{1}{q_2}\}$, $\Psi$ is a contraction on $\mathcal{X}_{\widetilde{T}}^+(2r)$. 
Thus, there exists a unique solution $\widetilde{v}$ to \eqref{eq:max_pr_BUC:pos_sol} on $[0,\widetilde{T}]$.
By \eqref{eq:max_pr_BUC:pos_sol}, \eqref{eq:max_pr_BUC:Psi},
\begin{equation}\label{result}
	\widetilde{v}(x,t)\ge v(x,0)\ge0,\qquad x\in\X,\ t\in[0,\widetilde{T}].
\end{equation}
By the considerations above,  $0\leq w\in C_b([0,T]\to E)$ yields $0 \leq \Theta(s, w(x,s))\in C_b([0,T]\to E)$.
Hence $\widetilde{v}$ is a solution to \eqref{eq:max_pr_BUC:u2-u1_lin} on $[0,\widetilde{T}]$ as well. Namely,
\[
\widetilde{v}(x,t)=v(x,0)+\int_0^t \Theta(s,\widetilde{v}(x,s))\,ds=:\Xi(\widetilde{v})(x,t).
\]
By the same arguments as the above, $\Xi$ is a contraction on $\mathcal{X}_{\widetilde{T}}(2r)$, for the same $\widetilde{T}$.
We deduce that $v=\widetilde{v}$ on $\X \times[0,\widetilde{T}]$.
Then, by \eqref{result}, $v(x,t)\geq0$ on $\X \times[0,\widetilde{T}]$, that yields 
\[
	0\leq u_1(x,\widetilde{T}) \leq u_2(x,\widetilde{T})\leq r,\ x\in \X.
\]
In the same way, the proof can be extended on $[\widetilde{T},2\widetilde{T}], [2\widetilde{T},3\widetilde{T}]$, \ldots, keeping the same $q_1$ and $q_2$, and, therefore, on the whole $[0,T]$.
Then $v(x,t)\geq 0$ on $\X\times[0,T]$, that yields the statement.
\end{proof}

Clearly, Theorem~\ref{thm:compar_pr} in the case $r=\theta$, $Hv = \ka a*v - vGv - mv$, $v\in E$, implies the first statement of Theorem~\ref{thm:compar_pr_basic}. The following Proposition yields the second statement of  Theorem~\ref{thm:compar_pr_basic}.
\begin{prop}\label{prop:solution_stays_in_tube}
	Let \eqref{assum:kappa>m}--\eqref{assum:sufficient_for_comparison} hold and $0\leq u_0 \leq \theta$.
	Then there exists a unique (classical) solution $u$ to \eqref{eq:basic}, and $0\leq u(x,t)\leq \theta$ for any $x\in\X$, $t\geq0$.
\end{prop}
\begin{proof}
	We set $Hv := \ka a*v - vGv - mv$ for $v\in E^+_\theta$, and $ Hv := H(0\vee v \wedge \theta)$ for $v\in E\setminus E^+_\theta$. Prove, first, that $H$ is (globally) Lipschitz continuous on $E$. 
	Indeed, for any $x,y\in\R$, 
	\[
	\lvert x\wedge \theta -y\wedge \theta\rvert=
	\frac{1}{2}\bigl\lvert (x+\theta-|x-\theta|)-(y+\theta-|y-\theta|)\bigr\rvert
	\leq |x-y|
	\]
	and, similarly, $|x\vee 0-y\vee 0|\leq |x-y|$.
Therefore, denoting $v_\theta:=0\vee v \wedge \theta$ for $v\in E$, one gets that $\lVert v_\theta-w_\theta\rVert_E \leq \lVert v-w\rVert_E$ for $w,v\in E$, and hence 
\begin{align*}
	\|Hw-Hv\|_E &\leq (\ka  +m +\sup_{v\in E} \|G(0\vee v \wedge \theta)\|_E + \theta \lt  ) \|w-v\|_E\\&=  (2\ka + \theta \lt) \|w-v\|_E.
\end{align*}
	As a result, for any $T>0$, the initial value problem
	\[
		\frac{\partial \widetilde{u}}{\partial t}(x,t) = (H\widetilde{u})(x,t),\quad \widetilde{u}(x,0)=u_0(x), \quad x\in \X,\ t\in (0,T],
	\]
	has a unique classical solution $\widetilde{u}$, i.e., for $\mathcal{F}_\theta$ defined by \eqref{Foper}, $\mathcal{F}_\theta \widetilde{u}\equiv 0$.
	
	Note that, for any $r\geq\theta$, $v\in E_r^+$ implies $Hv=H(v\wedge \theta)$. In particular, applying this for $v=0\vee \widetilde{u}\wedge r$, one gets
	\begin{equation}\label{eq:evidentbut}
	\mathcal{F}_r \widetilde{u}=\mathcal{F}_\theta \widetilde{u}\equiv 0.
	\end{equation}
	Moreover, by \eqref{assum:sufficient_for_comparison}, there exists $p\geq0$, such that, for any $r\geq\theta$, $v,w\in E_r^+$, $v \leq w$,
	\[
		p(w-v) + Hw-Hv \geq p(w\wedge\theta-v\wedge\theta)+ H(w\wedge \theta) - H(v\wedge \theta) \geq 0.
	\]
	Assume that $\|\widetilde{u}\|_T>\theta$. Then, by the arguments above and \eqref{eq:evidentbut}, we may apply Theorem~\ref{thm:compar_pr} for the case $r=\|\widetilde{u}\|_T$, $u_1\equiv 0$, $u_2= \widetilde{u}$ (note that, evidently, $\mathcal{F}_r0=0$). It yields $\widetilde{u}\geq 0$.
	Next, similarly, we can apply Theorem~\ref{thm:compar_pr} for the case $r=\theta$, $u_1 =\widetilde{u}$, $u_2\equiv \theta$ (since $\mathcal{F}_\theta \theta=0$). It~implies then that $\widetilde{u}\leq \theta$, that contradicts the assumption, therefore, $\|\widetilde{u}\|_T\leq \theta$. Apply once more Theorem~\ref{thm:compar_pr} for the case $r=\theta$, $u_1\equiv 0$, $u_2= \widetilde{u}$, then $\widetilde{u}\geq 0$.
	As a result, the function $\widetilde{u}=0\vee \widetilde{u}\wedge \theta$ solves \eqref{eq:basic}.
	
	Choose an arbitrary extension of $G$ on $\{0\leq v\in E\}$ such that \eqref{eq:G_loc_lip} holds.
	By Theorem~\ref{thm:exist_uniq_basic}, there exists a unique classical solution $u$ to \eqref{eq:basic}.
	Hence $0\leq u=\widetilde{u}\leq \theta$.
	The proof is fulfilled.
\end{proof}

\section{\protect The hair-trigger effect: \normalfont proofs of Theorems~\ref{thm:ht1myes}, \ref{thm:ht1mno}}\label{sec:hair_trigger}
 
We are going to prove our main Theorems~\ref{thm:ht1myes} and \ref{thm:ht1mno}.
The Section is organized as follows. First, in Propositions~\ref{prop:u_in_BUC}--\ref{prop:u_gr_0}, we show some properties of solutions to \eqref{eq:basic} with continuous initial conditions. Note that, by existence and uniqueness Theorem~\ref{thm:exist_uniq_basic}, the solutions will be also continuous and, moreover, by comparison Theorem~\ref{thm:compar_pr_basic},
any solution in $E=L^\infty(\X)$ can be estimated from above and below by continuous ones taking the corresponding estimates for the initial condition $u_0\not\equiv0$, cf.~Remark~\ref{rem:notequiv}. 

Next, we describe general Weinberger's scheme \cite{Wei1982a} for a dynamical system in discrete time in the context of the equation \eqref{eq:basic} (Propositions~\ref{prop:Q_def} and~\ref{prop:monot_along_vector_sol}, Lemma~\ref{lem:conv_to_theta-1}), and prove the corresponding result for continuous time (Proposition~\ref{prop:hairtrigger_general}). The latter result is proved under additional assumptions inherited by general Weinberger's approach: a technical assumption \eqref{hyp:descrete_front_nonempty} on the dynamical system and an assumption \eqref{eq:init_cond_is_large} on the initial condition $u_0$, which cannot be verified for particular examples of $u_0$, cf.~Remark~\ref{rem:rsigmaisbad}. 

Then, in Proposition~\ref{prop:suff_H1}, by using Lemma~\ref{lem:average_of_jump_gen_is_zero}, we prove that the technical assumption \eqref{hyp:descrete_front_nonempty} holds. To get rid of restrictions on initial condition $u_0$, one needs more machinery. Namely, we find in Proposition~\ref{prop:subsolutiontolinear} a useful sub-solution to the linearization of the equation \eqref{eq:basic} around the zero solution. Next, we show that (being multiplied on a small enough constant) it will be a sub-solution to the nonlinear equation \eqref{eq:basic} as well (Proposition~\ref{prop:subsolution}) and, in Proposition~\ref{prop:useBrandle}, we show that a solution to \eqref{eq:basic} becomes larger than the sub-solution after a big enough time. As a result, one can show that Weinberger's  assumption \eqref{eq:init_cond_is_large}  on the initial condition is fulfilled (just starting from a moment of time $t_0>0$ rather than from $0$). Finally, in the proof of Theorem~\ref{thm:ht1mno}, we show how to deal with the kernels without the first moment (where the assumption~\eqref{assum:first_moment_finite} fails).

\begin{prop}\label{prop:u_in_BUC}
  Let $0\leq u_0\in \Buc$, and suppose that $u$ is the corresponding classical solution to \eqref{eq:basic}. 
  Suppose also, that there exists $C>0$, such that
  \[
     0 \leq u(x,t) \leq C,\quad x\in\X ,\ t\ge0,
  \]
  and $g_C:=\sup\limits_{v\in E_C^+} |Gv|<\infty$. 
  Then $u\in C_{ub} (\X\times\R_+)$ and, moreover, $\|u(\cdot,t)\|_E\in C_{ub}(\R_+)$. In particular, these inclusions hold if we assume \eqref{assum:kappa>m}--\eqref{assum:sufficient_for_comparison}.
\end{prop}
\begin{proof}
  Being classical solution to \eqref{eq:basic}, $u$ satisfies the integral equation
  \[
    u(x,t) = u_0(x) + \int _0^t\bigl(\ka(a*u)(x,s) - mu(x,s) - u(x,s) (Gu)(x,s)\bigr)\,ds.
  \]
  Hence for any $x,y\in\X$, $0\leq\tau< t$, one has
  \[
  	\bigl|u(x,t)-u(y,\tau)\bigr|\leq (2\ka C + 2mC +Cg_C)(t-\tau),
  \]
  that fulfills the proof of the first inclusion. Then, the second one follows from the inequality $\bigl\lvert \lVert u(\cdot,t)\rVert_E - \lVert u(\cdot,\tau)\rVert_E \bigr\rvert\leq \lVert u(\cdot,t)-u(\cdot,\tau)\rVert_E $, $t,\tau\in\R_+$.
  Finally, if the conditions \eqref{assum:kappa>m}--\eqref{assum:sufficient_for_comparison} hold, then, by Proposition~\ref{prop:solution_stays_in_tube}, one gets that the solution $u$ exists and satisfies the conditions above if only $C:=\theta$. Moreover, \eqref{assum:Glipschitz} implies that, for any $v\in E_\theta^+$,
  \[
  	\|G v \|_E\leq \|G0\|_E+\lt \|v\|_E\leq \|G0\|_E+\theta \lt<\infty,
  \]
  that fulfills the proof.
\end{proof}

The maximum principle is a `standard counterpart' of the comparison principle, see e.g. \cite{Cov2007}. We will use in the sequel that, under some additional assumptions, the solutions to \eqref{eq:basic} are strictly positive; this is a quite common feature of linear parabolic equations, however, in general, it may fail for nonlinear ones. Consider the corresponding statement.
\begin{prop}\label{prop:u_gr_0}
	Let $E=\Cb$. Let \eqref{assum:kappa>m}--\eqref{assum:a_nodeg} hold with $G:E\to E$, such that $Gl\not\equiv \beta$, for $l\in(0,\theta)$. (In particular, the latter holds, if we assume, additionally, \eqref{assum:G_commute_T}--\eqref{assum:G_increas_on_const}.)
	Let $u_0\in E_\theta^+$, $u_0\not\equiv\theta$, $u_0\not\equiv0$, be~the~initial condition to \eqref{eq:basic} and $u$ be the corresponding solution. Then
  \[
  	u(x,t)>\inf_{\substack{y\in\X\\ s>0}}u(y,s)\geq0, \qquad x\in\X, t>0.
  \]
\end{prop}
\begin{proof}
	By Proposition~\ref{prop:solution_stays_in_tube}, $0\leq u(x,t)\leq \theta$, $x\in\X$, $t\geq0$. We denote
	\begin{equation}\label{jump}
		(L_a u)(x,t) = \ka (a*u)(x,t) - \ka u(x,t).
	\end{equation}
	Then, by \eqref{assum:Gpositive},
\begin{equation}\label{hintineq}
  \dfrac{\partial u}{\partial t}(x,t)- (L_a u)(x,t) = u(x,t) (\beta-(Gu)(x,t)) \ge0.
\end{equation}
Prove that, under \eqref{hintineq}, $u$ cannot attain its infimum on $\X\times(0,\infty)$ without being a constant. Indeed, suppose that, for some $x_0\in\X$, $t_0>0$,
\begin{equation}\label{minleq}
u(x_0,t_0)\leq u(x,t), \quad x\in\X, t>0.
\end{equation}
Then, clearly,
\begin{equation}\label{minpoint}
  \dfrac{\partial u}{\partial t}(x_0,t_0)=0,
\end{equation}
and \eqref{hintineq} yields
$(L_a u)(x_0,t_0)\le0$. On the other hand, \eqref{jump} and \eqref{minleq} imply
$(L_a u)(x_0,t_0)\ge0$. Therefore,
\[
  \int_\X a(x_0-y)(u(y,t_0)-u(x_0,t_0)) \,dy=0.
\]
Then, by \eqref{assum:a_nodeg}, for all $y\in B_\varrho(x_0)$,
\begin{equation}
u(y,t_0)=u(x_0,t_0).\label{eq:getmax}
\end{equation}
By the same arguments, for an arbitrary $x_{1}\in\partial B_{\varrho}(x_0)$,
we obtain \eqref{eq:getmax}, for all $y\in B_{\varrho}(x_{1})$.
Hence, \eqref{eq:getmax} holds on $B_{2\varrho(x_0)}$, and so on. As a result, \eqref{eq:getmax}~holds, for all $y\in\X $, thus $u(\cdot,t_0)$ is a constant, i.e.
\[
u(x,t_0)=u(x_0,t_0)=:l_0\in[0,\theta], \quad x\in\X.
\]
Then, considering \eqref{eq:basic} at $(x_0,t_0)$, and taking into account \eqref{minpoint}, one gets 
\[
0=u(x_0,t_0) \bigl(\beta - (Gu) (x_0,t_0)\bigl)=l_0(\beta- Gl_0).
\]
By the assumption, the latter equality is possible if only $l_0\in\{0,\theta\}$, i.e. either $u(\cdot,t_0) \equiv 0$ or $u(\cdot, t_0)\equiv \theta$.
By \eqref{minleq}, $u(x_0,t_0)=\theta\geq\sup_{y\in\X,s>0}u(y,s)$ implies $u\equiv\theta$, that contradicts $u_0\not\equiv\theta$. 
Hence $u(x,t_0)=u(x_0,t_0)=0$, $x\in\X$. 
Then, by Theorem~\ref{thm:exist_uniq_BUC}, $u(x,t)=0$, $x\in\X$, $t\geq t_0$.
And now one can consider the reverse time in \eqref{eq:basic} starting from $t=t_0$. Namely, we set $w(x,t):=u(x,t_0-t)$, $t\in[0,t_0]$, $x\in\X$. Then $w(x,0)=u(x,t_0)=0$, $x\in\X$, and
\begin{equation}
	\dfrac{\partial w}{\partial t}(x,t)= w(x,t)(Gw)(x,t) -\ka^{+}(a^{+}*w)(x,t) +mw(x,t).
\label{eq:inverse_basic}
\end{equation}
Prove that the equation \eqref{eq:inverse_basic} with the initial condition $w(\cdot,0)\equiv 0$ has a unique classical solution $w \equiv 0$ in $C_b([0,t_0]\to E)$. 
Indeed, let $w\in C_b([0,t_0]\to E)$ solve \eqref{eq:inverse_basic}. Suppose that the set 
\[
	K:=\bigl\{t\in[0,t_0]\bigm\vert \|w(\cdot,t)\|_E >0\bigr\}
\]
is not empty, i.e. $w \not\equiv 0$. We define then $T:=\inf K$. 
In particular, $\|w(\cdot,t)\|_E=0$ for $t\in[0,T)$ (note that the latter interval might be empty if $T=0$). Since the function $\tau\mapsto \|w(\cdot,\tau)\|_E$  is continuous, we have that $\|w(\cdot,T)\|_E=0$ as well. 
Therefore, $T=t_0$ would contradict the assumption $K\neq\emptyset$; hence $T<t_0$.
Consider now the equation \eqref{eq:inverse_basic} for $t\in[T,t_0]$ with the initial value $w(\cdot,T)\equiv 0$. 
It is straightforward to check that the assumptions on $G$ imply that, for any $r>0$, there exists $\Delta T>0$, such that $T+\Delta T<t_0$ and the mapping
\[
	\Psi(w)(x,t) = \int_T^{T+t} w(x,s) (Gw)(x,s) -\ka (a*w)(x,s) +m w(x,s)ds.
\]
is a contraction on $C_b([0,\Delta T]\to E)$. Therefore, by the uniqueness arguments, $w(\cdot,t)\equiv 0$  for $t\in [T,T+\Delta T]$ that contradicts the choice of $T$. Therefore, $K=\emptyset$, i.e. $w(\cdot,t)\equiv 0$ for all $t\in[0,t_0]$, in particular, $u(\cdot,0)=w(\cdot,t_0)\equiv 0$, that contradicts $u_0\not\equiv0$. Thus, the initial assumption was wrong, and \eqref{minleq} can not hold.
The proof is fulfilled.
\end{proof}

In the sequel, it will be useful to consider the solution to \eqref{eq:basic} as a nonlinear transformation of the initial condition.
\begin{defn}
	For a fixed $t>0$, define the mapping $Q_{t}$ on $\{f\in E\mid f\geq0\}$ by
  \begin{equation}
  (Q_{t}f)(x):=u(x,t),\quad x\in\X,\label{def:Q_T}
  \end{equation}
  where $u(x,t)$ is the solution to \eqref{eq:basic} with the initial condition $u(x,0)=f(x)$.
\end{defn}
Let us collect several properties of $Q_t$ needed below.
\begin{prop}
  \label{prop:Q_def}
	Let \eqref{assum:kappa>m}--\eqref{assum:G_increas_on_const} hold. Then, for any fixed $t>0$, the mapping $Q:=Q_{t}:\{f\in E\mid f\geq0\}\to\{f\in E\mid f\geq0\}$ satisfies the following properties
  \begin{enumerate}[label=\textnormal{(Q\arabic*)}]
    \item $Q: E_\theta^+ \to E_\theta^+$; \label{eq:QBtheta_subset_Btheta}
    \item let $T_y$, $y\in\X$, be a translation operator, given by \eqref{shiftoper}, then \label{prop:QTy=TyQ}
        \[
          (Q T_{y}f)(x)=( T_{y}Qf)(x), \quad x,y\in\X,\ f\in E_\theta^+;
        \]
    \item $Q0=0$, $Q\theta=\theta$, and $Q r>r$, for any constant $r\in(0,\theta)$; \label{prop:Ql_gr_l}
    \item if $f,g\in E_\theta^+$ and $f \leq g $, then $Qf \leq Qg$;\label{prop:Q_preserves_order}
    \item if $f_{n}\locun f$, then $(Qf_{n})(x)\to (Qf)(x)$ for (a.a.) $x\in\X$.\label{prop:Q_cont}
  \end{enumerate}
\end{prop}
\begin{proof}
	The property~\ref{eq:QBtheta_subset_Btheta} follows from Proposition~\ref{prop:solution_stays_in_tube}.
	To prove~\ref{prop:QTy=TyQ} we note that, by \eqref{assum:G_commute_T}, $ T_yG=G T_y$, and $ T_y(a*u)=a*( T_y u)$, and then,  by \eqref{eq:exist_uniq_BUC:B}, $B( T_y v)= T_y(Bv)$. Using further the notations in the proof of Theorem~\ref{thm:exist_uniq_BUC}, we will proceed by the induction in $n$. Namely, assume that $Q_t T_y=T_y Q_t$ for $t\in[0,\Upsilon_{n-1}]$. 
Denote $\Phi_\tau[u_\tau]:=\Phi_\tau$, given by \eqref{eq:exist_uniq_BUC:Phi_v} (to specify the dependence on the initial condition $u_\tau$). Then $T_y (\Phi_\tau[u_\tau] v)=\Phi_\tau[T_yu_\tau] (T_y v) $ for all $v\in \x_{\tau,\Upsilon}^+(r_{n})$, where $[\tau,\Upsilon]:=[\Upsilon_{n-1},\Upsilon_n]$. Then, for $t\in(\tau,\Upsilon]$,
\begin{align*}
Q_tT_y f &= \lim_{N\rightarrow \infty} \bigl(\Phi_\tau[Q_\tau T_y f]\bigr)^N \bigl(T_y  v(\cdot,t)\bigr)= \lim_{N\rightarrow \infty} \bigl(\Phi_\tau[T_y Q_\tau  f]\bigr)^N \bigl(T_y  v(\cdot,t)\bigr)\\& = \lim_{N\rightarrow \infty} T_y (\Phi_\tau[Q_\tau f])^N v(\cdot,t) = T_y  Q_{t} f.
\end{align*}

By (Q2), $u_0(x)\equiv r\in(0,\theta)$ yields $u(\cdot,t) = \mathrm{const}$, $t\geq0$.
Then, by \eqref{assum:Gpositive} and \eqref{assum:G_increas_on_const}, for any $t\geq0$, we have
\[
	Qr = u(t) = r + \int_0^t u(s)(\beta-(Gu)(s))ds > 0.
\]
Hence the property~\ref{prop:Ql_gr_l} holds.
The property~\ref{prop:Q_preserves_order} holds by Theorem~\ref{thm:compar_pr}. The property~\ref{prop:Q_cont} is a weaker version of Proposition~\ref{prop:loc_uni}.
\end{proof}
	\begin{rem}\label{rem:monostability}
	Take an arbitrary constant $r\in(0,\theta)$. One can treat then $r$ as a constant function from $E_\theta^+$.
  	By \ref{prop:Ql_gr_l} and \ref{prop:Q_preserves_order}, the sequence $\bigl( Q_t^n r\bigr)_{n\geq1}\subset (0,\theta]$ is non-decreasing for an arbitrary $t>0$.
  	Hence there exists the limit $r_\infty:=\lim\limits_{n\to\infty} Q_t^n r\in(0,\theta]$. Next, by \ref{prop:Q_cont},
  	\[
  		Q_tr_\infty = Q_t \lim_{n\to\infty}  Q_t^n r = \lim_{n\to\infty} Q_t^{n+1}r = r_\infty.
  	\]
  	Hence, by \ref{prop:Ql_gr_l}, $r_\infty=\theta$.
  	By Proposition \ref{prop:u_in_BUC}, $Q_t r$ is uniformly continuous in~$t>0$. 
		As a result, $\lim\limits_{t\to\infty} Q_t r = \theta$.	
		Therefore, by \ref{prop:Q_preserves_order}, for any $u_0\in E$ with $0< r \leq u_0 \leq \theta$, we have
		\[
			\theta = \lim_{t\to\infty} Q_tr \leq \lim_{t\to\infty} Q_t u_0 \leq \theta, \qquad \text{ and hence } \lim_{t\to\infty} Q_t u_0 = \theta.
		\]
		As a result, $u\equiv 0$ is unstable and $u\equiv \theta$ is asymptotically stable solutions to \eqref{eq:basic} in $E_\theta^+$. For this reason, we refer to \eqref{eq:basic} as to a monostable-type equation.
	\end{rem}

Let $\S$ denote a unit sphere in $\X$ centered at the origin:
\[
\S =\bigl\{x\in\X\bigm| |x|=1\bigr\};
\]
in particular, $S^{0}=\{-1,1\}$.

\begin{defn}\label{def:monotoneindirection}
A function $f\in E$ is said to be increasing (decreasing, constant) along the vector $\xi\in\S $ if, for a.a. $x\in\X $, the function $f(x+s\xi)=( T_{-s\xi}f)(x)$ is increasing (decreasing, constant) in $s\in\R$, respectively.
\end{defn}

\begin{prop}\label{prop:monot_along_vector_sol}
	Let \eqref{assum:kappa>m}--\eqref{assum:G_increas_on_const} hold.
  Let $u_0\in E_\theta^+$ be the initial condition for the equation \eqref{eq:basic} which is
	increasing (decreasing, constant) along a vector $\xi\in\S $; and $u(\cdot,t)\in E_\theta^+$, $t\geq0$, be the corresponding solution (cf. Proposition~\ref{prop:solution_stays_in_tube}).
  Then, for any $t>0$, $u(\cdot,t)$ is increasing (decreasing, constant, respectively) along the $\xi$.
\end{prop}
\begin{proof}
Let $u_0$ be decreasing along a $\xi\in\S $.
Take any $s_1\leq s_2$ and consider two initial conditions to \eqref{eq:basic}: $u_0^i(x)=u_0(x+s_i\xi)=( T_{-s_i\xi}u_0)(x)$, $i=1,2$ (cf. \eqref{shiftoper}).
Since $u_0$ is decreasing, $u_0^1(x)\geq u_0^2(x)$, $x\in\X$.
Then, by Proposition~\ref{prop:Q_def}, 
\[
	 T_{-s_1\xi}Q_tu_0=Q_t T_{-s_1\xi}u_0=Q_t u_0^1\geq Q_tu_0^2=Q_t T_{-s_2\xi}u_0=
	 T_{-s_2\xi}Q_tu_0,
\]
that proves the statement. The cases of a decreasing $u_0$ can be considered in the same way. The constant function along a vector is decreasing and decreasing simultaneously.
\end{proof}

To prove the hair-trigger effect (Theorems~\ref{thm:ht1myes}, \ref{thm:ht1mno}), we will follow the abstract scheme proposed in \cite{Wei1982a} for a dynamical system in discrete time. Note that all statements there were considered in the space $E=C_{b}(\X)$.

Consider the set $N_\theta$ of all non-increasing functions $\varphi\in C(\R)$, such that
$\varphi(s)=0$, $s\geq0$, and
\begin{equation*}
\varphi(-\infty):=\lim _{s\to-\infty}\varphi(s)\in(0,\theta).\label{def:fy_Weinberger}
\end{equation*}

For arbitrary $s\in\R$, $c\in\R$, $\xi\in\S $,
define the following  mapping $V_{s,c,\xi}:L^\infty(\R)\to L^\infty(\X)$
\begin{equation}\label{defofV}
  (V_{s,c,\xi}g)(x)=g(x\cdot \xi+s+c), \quad x\in\X.
\end{equation}
Fix an arbitrary $\varphi\in N_\theta$.
 For $t>0$, $c\in\R$, $\xi\in\S $, consider the mapping $R_{t,c,\xi}:\ L^{\infty}(\R)\to L^{\infty}(\R)$, given by
\begin{equation}
	(R_{t,c,\xi}g)(s)=\max\bigl\{ \varphi(s),(Q_t (V_{s,c,\xi}g))(0)\bigr\},\quad s\in\R,\label{eq:iterop_by_Weinberger}
\end{equation}
where $Q_t:E\to E$ is a mapping which satisfies the conditions (Q1)--(Q5) in Proposition~\ref{prop:Q_def} (in particular, one can consider $Q_t$ given by \eqref{def:Q_T} provided that  \eqref{assum:kappa>m}--\eqref{assum:G_increas_on_const} hold).
Consider now the following sequence of functions
\begin{equation}\label{fiteration}
  f_{n+1}(s)=(R_{t,c,\xi}f_n)(s),\quad f_0(s)=\varphi(s),\qquad s\in\R, n\in\N\cup\{0\}.
\end{equation}
By Proposition~\ref{prop:Q_def} and \cite[Lemma~5.1]{Wei1982a}, $0\leq \phi(s) \leq \theta$, $s\in\R$, implies $0\leq f_n(s) \leq f_{n+1}(s) \leq \theta$, $s\in\R$, $n\in\N$; 
hence one can define the following limit
\begin{equation}
	f_{t,c,\xi}(s):= \lim_{n\to\infty}f_n(s), \quad s\in\R.\label{eq:limit_func_Weinberger}
\end{equation}
Also, by \cite[Lemma~5.1]{Wei1982a}, for fixed $\xi\in\S $, $t>0$, $n\in\N$, the functions $f_n(s)$ and $f_{t,c,\xi}(s)$ are non-increasing in $s$ and in $c$; moreover, $f_{t,c,\xi}(s)$ is a lower semi-continuous function of $s,c,\xi$, as a result, this function is continuous from the right in $s$ and in $c$. Note also, that $0\leq f_{t,c,\xi}\leq\theta$.
Then, for any $c,\xi$, one can define the limiting value
\[
f_{t,c,\xi}(\infty):=\lim_{s\to\infty}f_{t,c,\xi}(s).
\]
Next, for any $t>0$, $\xi\in\S $, we define
\begin{equation*}
c_t^{*}(\xi)=\sup\{ c\mid f_{t,c,\xi}(\infty)=\theta\} \in\R\cup\{-\infty,\infty\},
\end{equation*}
where, as usual, $\sup\emptyset:=-\infty$. By \cite[Propositions~5.1, 5.2]{Wei1982a}, one has
\begin{equation}\label{jumpfunc}
  f_{t,c,\xi}(\infty)=\begin{cases}
    \theta, & c<c_t^*(\xi),\\
    0, & c\geq c_t^*(\xi),
  \end{cases}
\end{equation}
cf. also \cite[Lemma~5.5]{Wei1982a}; moreover, $c_t^*(\xi)$ is a lower semi-continuous function of~$\xi$. It is crucial that, by \cite[Lemma 5.4]{Wei1982a}, neither $f_{t,c,\xi}(\infty)$ nor $c_t^{*}(\xi)$ depends on the choice of $\varphi\in N_\theta$. Note that the monotonicity of $f_{t,c,\xi}(s)$ in $s$ and \eqref{jumpfunc} imply that, for $c<c_t^*(\xi)$, $f_{t,c,\xi}(s)=\theta$, $s\in\R$.

Define
\begin{equation}\label{eq:TauT}
	\Upsilon_{t}:=\bigl\{ x\in\X \bigm\vert x\cdot\xi\leq c_t^{*}(\xi), \xi\in \S \bigl\}, \quad t>0.
\end{equation}
For $A\subset\X$, $x\in\X$, $s\in\R$, we denote also
\[
x+A:=\{x+y\mid y\in A\}\subset\X, \quad sA:=\{sy\mid y\in A\}\subset\X.
\]
We will need the following Weinberger's result:
\begin{lem}[{cf. \cite[Theorem~6.2]{Wei1982a}}] \label{lem:conv_to_theta-1}
  Let $E=\Cb$ and $v_0\in E_\theta^+$. Let, for some fixed $t>0$, $Q=Q_t:E\to E$  be a mapping which satisfies the conditions (Q1)--(Q5) in Proposition~\ref{prop:Q_def}, and $\Upsilon_t$ be defined by \eqref{eq:TauT}. Suppose that 
  \begin{equation}\label{eq:upsisnotemp}
  \inter(\Upsilon_t)\neq\emptyset.
  \end{equation}
  Then, for any compact set $\Tauin_t\subset\inter(\Upsilon_t)$ and for any $\sigma\in(0,\theta)$, one can choose a radius $r_\sigma=r_\sigma(Q_t,\Tauin_t)>0$, such that, for any fixed $x_0\in\X$,
  \begin{equation}\label{initcondissepfrom0}
    v_0(x)\geq\sigma, \quad x\in B_{r_\sigma}(x_0),
  \end{equation}
  implies
  \begin{equation}\label{liminfconv}
    \lim_{n\rightarrow\infty}\min\limits _{x\in n\Tauin_t} Q_t^n v_{0}(x)=\theta.
  \end{equation}
\end{lem}
\begin{rem}\label{rem:rsigmaisbad}
Note that, in \cite[Theorem~6.2]{Wei1982a}, the existence of $r_\sigma$ is proved only; there are not any estimates on it. As a result, for a given $v_0\in E^+_\theta$, the condition \eqref{initcondissepfrom0} cannot be checked directly. 
\end{rem}
\begin{rem}
	There is no loss of generality if we assume that \eqref{initcondissepfrom0} holds for $x_0{=}\,0$ only.
	Indeed, for any $x_0\in\X$, $\Tauin_t\subset \inter (\Upsilon_t)$, there exist $N=N(x_0,\Tauin_t)$, $\widetilde{\Tauin}_t\subset \inter (\Upsilon_t)$, such that, for all $n\geq N$, one gets $x_0+n\Tauin_t\subset n\widetilde{\Tauin}_t$.
Therefore, we have 
\begin{align*}
	\theta &\geq \lim_{n\to\infty} \min_{x\in n\Tauin_t} (Q_t^n T_{-x_0} u_0)(x) = \lim_{n\to\infty} \min_{x\in n\Tauin_t} (T_{-x_0} Q_t^n u_0) (x)\\
	&= \lim_{n\to\infty} \min_{x\in x_0+ n\Tauin_t} (Q_t^n u_0)(x) \geq \lim_{n\to\infty} \min_{x\in n\widetilde{\Tauin}_t} (Q_t^n u_0)(x) = \theta.
\end{align*}
\end{rem}

The following statement presents a counterpart of Lemma~\ref{lem:conv_to_theta-1} for continuous time provided that the mapping $Q_t$ is given by the solution to \eqref{eq:basic} as in \eqref{def:Q_T}.
\begin{prop}\label{prop:hairtrigger_general}
	Let \eqref{assum:kappa>m}--\eqref{assum:G_increas_on_const} hold and $u_0\in \Buc$. Let $Q_t$, $t>0$, be given by \eqref{def:Q_T}, and let the corresponding $\Upsilon_t$, $t>0$, be given by \eqref{eq:TauT}.
Suppose that, for some compact $\Tauin\!\subset\!\inter(\Upsilon_1)$, there exists 
$\n\!\in\!\inter(\Tauin)$, such that
\begin{equation}\label{hyp:descrete_front_nonempty}
	\frac{1}{j}\n \in \inter(\Upsilon_{\frac{1}{j}}),\qquad j\in\N.  
\end{equation}
Let $\sigma\in(0,\theta)$ and $r_\sigma=r_\sigma(Q_1,\Tauin)$ be chosen according to Lemma~\ref{lem:conv_to_theta-1}. Suppose that
\begin{equation}\label{eq:init_cond_is_large}
	u_0(x) \geq \sigma, \qquad x\in B_{r_\sigma (Q_1, \Tauin)}.
\end{equation}
	Then, for the corresponding solution $u$ to \eqref{eq:basic} and for any compact $K\subset \X$, the following limit holds
	\begin{equation}\label{eq:hairtrigger_general}
		\min_{x\in K} u(x+t \n,t) \to \theta, \qquad t\to \infty.
	\end{equation}
\end{prop}

\begin{proof}
	First, we note that, by Proposition~\ref{prop:Q_def}, the conditions (Q1)--(Q5) hold for all $Q=Q_t$, $t>0$.
	We denote $K_1:=-\n + \Tauin$. Because of \eqref{eq:init_cond_is_large}, one can apply Lemma~\ref{lem:conv_to_theta-1} for $t=1$ and $v_0(x):=u_0(x)$, $x\in\X$. Namely, since $Q_1^nv_0(y)=Q_1^nu_0(y)=u(y,n)$, $y\in\X$, one gets, by \eqref{liminfconv},
	\begin{equation}\label{eq:convatnatural}
	\min_{x\in n K_1} u(x+n \n,n)=\min_{y\in n\Tauin}u(y,n)\to\theta, \quad n\to\infty.
	\end{equation}
	Next, by~\eqref{hyp:descrete_front_nonempty}, $0 \in -\tfrac{1}{2}\n + \inter(\Upsilon_{\frac{1}{2}})$.
	Choose now any compact $K_2\subset -\tfrac{1}{2}\n + \inter(\Upsilon_{\frac{1}{2}})$, such that $0\in\inter(K_2)$.
	By Lemma~\ref{lem:conv_to_theta-1} for $t=\frac12$ and $\Tauin_{\frac{1}{2}}:=K_2+\tfrac{1}{2}\n\subset \inter(\Upsilon_{\frac{1}{2}})$ there exists a radius
$r_\sigma(Q_{\frac{1}{2}},\Tauin_{\frac{1}{2}})>0$.
 By~\eqref{eq:convatnatural}, there exists $N_1\geq1$, such that, for all $n\geq N_1$, 
\begin{equation}\label{eq:sdads}
B_{r_\sigma(Q_{\frac{1}{2}},\Tauin_{\frac{1}{2}})}(0)\cup K\subset n K_1,
\qquad u(x+n \n,n) \geq \sigma,\quad x\in nK_1.
\end{equation}
 Set $S_1:=N_1$; by the latter inclusion and \eqref{eq:sdads}, one can apply Lemma~\ref{lem:conv_to_theta-1} for
$v_0(x):=u(x+S_1 \n,S_1)$, $x\in\X$. Then
\[
Q_{\frac{1}{2}}^nv_0(y)=u\Bigl(y+S_1 \n,S_1+\frac{n}{2}\Bigr), \quad y\in\X,
\]
and hence
\begin{align}
&\quad\min_{x\in nK_2} u\Bigl(x + \Bigl(S_1+\frac{n}{2}\Bigr)\n, S_1+ \frac{n}{2}\Bigr) \notag\\&=\min_{y\in n\Tauin_{\frac{1}{2}}}
		u\Bigl(y + S_1\n, S_1+ \frac{n}{2}\Bigr) \to \theta,\quad n\to\infty.
		\label{eq:qwqewwqe}
\end{align}
Similarly, choose a compact $K_3\subset -\tfrac{1}{3}\n+ \inter(\Upsilon_{\frac{1}{3}})$ with $0\in \inter(K_3)$, and consider Lemma~\ref{lem:conv_to_theta-1} with $t=\frac13$ and $\Tauin_{\frac{1}{3}}:=K_3+\tfrac{1}{3}\n\subset \inter(\Upsilon_{\frac{1}{3}})$. Then, there exists a radius $r_\sigma(Q_{\frac{1}{3}},\Tauin_{\frac{1}{3}})>0$, and, by \eqref{eq:qwqewwqe}, there exists $N_2\geq 2$ such that for all $n\geq N_2$, 
\[
B_{r_\sigma(Q_{\frac{1}{3}},\Tauin_{\frac{1}{3}})}\cup K \subset nK_2
\]
and
\[
	u\Bigl(x+\Bigr(S_1+\frac{n}{2}\Bigr)\n,S_1+\frac{n}{2}\Bigr) \geq \sigma,\quad x\in nK_2.
\]
Set $S_2:=S_1+\frac{N_2}{2}=N_1+\frac{N_2}{2}\geq2$ and apply Lemma~\ref{lem:conv_to_theta-1} with $v_0(x):=u(x+S_2\n,S_2)$, $x\in\X$. We have
	\[\min_{x\in nK_3} u\Bigl(x+\Bigl(S_2+\frac{n}{3}\Bigr)\n, S_2 + \frac{n}{3}\Bigr) =\min_{x\in n\Tauin_{\frac{1}{3}}} u\Bigl(x+S_2\n, S_2+ \frac{n}{3}\Bigr)\to \theta,\quad n\to\infty.
	\]
	By induction, for any $K_j\subset -\tfrac{1}{j}\n+ \inter(\Upsilon_{\frac{1}{j}})$, $j\geq 3$, with $0 \in \inter(K_j)$, one can set $\Tauin_{\frac{1}{j}}:=K_j+\tfrac{1}{j}\n\subset \inter(\Upsilon_{\frac{1}{j}})$ and choose  $N_{j-1}\geq j-1$ such that for all $n\geq N_{j-1}$, 
	\begin{gather*}
	B_{r_\sigma(Q_{\frac{1}{j}},\Tauin_{\frac{1}{j}})}\cup K \subset nK_{j-1}, \\
	u\Bigl(x+\Bigr(S_{j-2}+\frac{n}{j-1}\Bigr)\n, S_{j-2} + \frac{n}{j-1}\Bigr) \geq \sigma,\quad x\in nK_{j-1}.
	\end{gather*}
	Set 
	\[
	S_{j-1}:=S_{j-2}+\frac{N_{j-1}}{j-1}=N_1+\frac{N_2}{2}+ \ldots+ \frac{N_{j-1}}{j-1}\geq j-1.
	\] 
	Then, by Lemma~\ref{lem:conv_to_theta-1}, similarly to the above,
	\begin{equation}\label{eq:asdaw}
			\min_{x\in nK_j} u\Bigl(x+\Bigl(S_{j-1}+\frac{n}{j}\Bigr)\n,S_{j-1}+\frac{n}{j} \Bigr) \to \theta,\quad n\to\infty.
	\end{equation}
	
	Suppose that \eqref{eq:hairtrigger_general} does not hold. Then, for some $\eps>0$, there exist sequences $x_m\in K$, $m\in\N$, and $t_m\to\infty$ as $m\to\infty$, such that 
	\begin{equation}\label{eq:no_hairtrigger}
		u(x_m + t_m\n ,t_m)\leq \theta-\eps. 
	\end{equation}
	Since $u_0\in \Buc$, then by Proposition~\ref{prop:u_in_BUC},  $u\in C_{ub}(\X\times \R_+)$. Thus there exists $\delta=\delta(\eps)$, such that 
	\[
		\vert u(x,t) - u(y,s) \vert < \frac{\eps}{2},\qquad |x-y|<\delta,\ |t-s|<\delta.
	\]
	We choose $j\in\N$ such that $\max\{1,|\n|\} < \delta j$. 
	By \eqref{eq:asdaw}, there exists $N_j'>N_{j-1}$, such that, for all $n\geq N_j'$, we have that $K\subset n K_j$ and 
\[
\min_{x\in K} u\Bigl(x+\Bigl(S_{j-1}+\frac{n}{j}\Bigr)\n,S_{j-1}+\frac{n}{j} \Bigr)>\theta-\frac{\eps}{4}.
\]
Choose $m$, such that $t_m\geq S_{j-1}+\dfrac{N_j'}{j}$. Let $n_m$ be the entire part of $j(t_m-S_{j-1})$. Then
$n_m\geq N_j'$ and, for $q_m:=S_{j-1}+\dfrac{n_m}{j}$, we easily get that 
\[
\max\{1,|\n|\}\,|t_m-q_m|<\delta.
\]
Therefore,
	\begin{align*}
		u(x_m + &t_m\n ,t_m) \geq \min_{x\in K} u(x + t_m\n,t_m) \\
		&\geq \min_{x\in K} u(x+q_m\n,q_m) - \max_{x\in K} \big\vert 
		u(x+q_m\n,q_m) - u(x+t_m\n,t_m)\big\vert \\
			&\geq \min_{x\in K} u(x+q_m\n,q_m) -\frac{\eps}{2}>\theta - \frac{3}{4}\eps,
	\end{align*}
	that contradicts \eqref{eq:no_hairtrigger}. Therefore \eqref{eq:hairtrigger_general} holds and the proof is fulfilled. 
\end{proof}

We are going now to get rid of the assumptions \eqref{hyp:descrete_front_nonempty} and \eqref{eq:init_cond_is_large}  in Proposition~\ref{prop:hairtrigger_general}. We start with the following lemma.

\begin{lem}\label{lem:average_of_jump_gen_is_zero}
	Let $b\in L^1(\R \to \R_+)$ be such that 
\[
\int_\R b(s)\,ds=1,\qquad \int_\R  |s|b(s)ds<\infty, 
\]
and let $v\in L^\infty(\R\to \R_+)$ be a non-increasing function. Then the following limit holds
	\begin{equation}\label{eq:average_of_jump_gen_is_zero}
		\lim_{r\to\infty} \int_{-r}^r  \bigl( (b*v)(s)-v(s) \bigr) ds = \bigl( v(-\infty)-v(\infty) \bigr) \int_\R 
		s\, b(s)\, ds. 
	\end{equation}
\end{lem}
\begin{proof}
	For $r>0$ and $\varrho:=\dfrac{r}{2}$, we have, by Fubini's theorem,
	\begin{align*}
		&\quad \int_{-r}^r  (b*v)(s) ds - \int_{-r}^r  v(s)ds = \int_{-\infty}^\infty b(y) \int_{-r}^r  v(s-y)dsdy - \int_{-r}^r  v(s) ds \\
		&=\int_{-\infty}^\infty b(y) W_r(y)\,dy=I_1(r) +I_2(r),
	\end{align*}
where
	\begin{gather*}
	W_r(y):= \int_{-r-y}^{r-y} v(s)ds - \int_{-r}^r  v(s)ds , \quad y\in\R,\\
	I_1(r):= \int_{|y|\leq\varrho} b(y) W_r(y)\,dy, \qquad I_2(r):= \int_{|y|>\varrho} b(y) W_r(y)\,dy.
	\end{gather*}

	Clearly, 
	\[
		\sup_{|y|\leq \varrho} b(y) \lvert W_r(y) \rvert\leq 2\|v\|_E |y|b(y) \in L^1(\R).
	\]
	Next, because of the monotonicity of $v$, we have,
\begin{equation}\label{eq:monimplconv}
		 y v(-r)\leq \int_{-r-y}^{-r} v(s)ds\leq  y v(-r-y).
\end{equation}
	Since $-r-y<-\frac{r}{2}$ for $|y|\leq\varrho=\frac{r}{2}$, we have that 
	\[
		\1_{|y|\leq\varrho}\int_{-r-y}^{-r} v(s)ds \to v(-\infty)y,\qquad r\to\infty;
	\]
	and, similarly, $\1_{|y|\leq\varrho}\int_{r-y}^r v(s)ds \to v(\infty)y$.
	Therefore, by the dominated convergence theorem,
	\begin{equation}\label{eq:I1_lim}
		\lim_{r\to\infty} I_1(r) = \bigl(v(-\infty)-v(\infty)\bigr)\int_{|y|\leq\varrho}yb(y)dy.	
	\end{equation}
	On the other hand, 
	\begin{equation}\label{eq:I2_lim}
		|I_2(r)| \leq 2 v(-\infty)\, r \int_{|y|>\varrho} b(y)dy \leq 4 v(-\infty) \,\frac{r}{\varrho} \int_{|y|>\varrho} b(y)|y| dy \to 0,
	\end{equation}
	as $r\to\infty$. Combining \eqref{eq:I1_lim} and \eqref{eq:I2_lim}, one gets the statement.
\end{proof}

The following statement yields sufficient conditions for \eqref{hyp:descrete_front_nonempty}.
\begin{prop}\label{prop:suff_H1}
	Let \eqref{assum:kappa>m}--\eqref{assum:first_moment_finite} hold.
	Let $\Upsilon_t$, $t>0$, be defined by \eqref{eq:TauT}, and $\m$ be defined by \eqref{firstfullmoment}. Then
	\begin{equation}\label{eq:first_moment_in_the_front}
		 t\, \m \in \inter(\Upsilon_t).
	\end{equation}
\end{prop}
\begin{proof}
	Fix $t>0$. For a $\xi\in\S$, we set
	\begin{equation}\label{eq:setc}
	c:=\ka t\int_{\X} y\!\cdot\!\xi a(y)dy = t\,\m\!\cdot \!\xi\in\R.
	\end{equation}
Let $f_{t,c,\xi}$ be defined by \eqref{eq:limit_func_Weinberger}.
By the definition of $\Upsilon_t$ and \eqref{jumpfunc}, we have that if $f_{t,c,\xi}(\infty)=\theta$ for all $\xi\in\S$, then \eqref{eq:first_moment_in_the_front} holds. Suppose, in contrast, that, for some $\xi\in\S$, $f_{t,c,\xi}(\infty)=0$. Fix such a $\xi$, consider the corresponding $c$ according to \eqref{eq:setc}, and denote $f:=f_{t,c,\xi}$.
Note that, by \cite[Lemma 5.2]{Wei1982a} and the discussion thereafter, $f(-\infty)=\theta$.

We set $u_0(x):= f(x\cdot\xi)$, $x\in\X$, and consider the corresponding solution $u$ to \eqref{eq:basic}.
Then, by \eqref{defofV}, we evidently have 
\[
(V_{s,c,\xi}f)(x)=( T_{-(c+s)\xi}u_0)(x), \quad x\in\X.
\] 
Next, as it was mentioned above, the functions $f_n$ and $f=f_{t,c,\xi}$ in \eqref{eq:limit_func_Weinberger} are monotone, hence the limit in \eqref{eq:limit_func_Weinberger} is locally uniform. 
Therefore, passing $n$ to $\infty$ in \eqref{fiteration}, we will get from \eqref{eq:iterop_by_Weinberger} and Proposition~\ref{prop:Q_def}, that
\begin{align}\notag
f(s)&=\max\bigl\{\varphi(s),(Q_t(V_{s,c,\xi}f))(0) \bigr\}=\max\bigl\{\varphi(s),( T_{-(c+s)\xi}Q_tu_0)(0) \bigr\}\\&=\max\bigl\{\varphi(s),u\bigl((c+s)\xi,t\bigr) \bigr\}.\label{eq:ssadsad}
\end{align}
Since $f$ is non-increasing on $\R$, $u_0$ is non-increasing along $\xi$, cf.~Definition~\ref{def:monotoneindirection}; then, by Proposition~\ref{prop:monot_along_vector_sol}, $u$ also has the same property. As a result, the function
\begin{equation}\label{eq:defofvadd}
v(s):=u\bigl((c+s)\xi,t\bigr), \quad s\in\R
\end{equation}
is non-increasing on $\R$. Next, by our assumptions, $f(-\infty)=\theta>\varphi(-\infty)$ and $f(\infty)=0$; therefore,
 we get from \eqref{eq:ssadsad}, that
 \begin{equation}\label{eq:limit_of_f}
	\lim_{s\to\infty} v(s) = 0,\qquad \lim_{s\to-\infty} v(s) = \theta.
\end{equation}

Next, \eqref{eq:ssadsad} implies that, for each $s\in\R$, cf.~\eqref{eq:RDform},
\begin{align*}
u_0(s\xi)&\geq u\bigl((c+s)\xi,t\bigr)\\&=u_0\bigl((c+s)\xi\bigr)
+\int_0^t \ka\Bigl((a*u)\bigl((c+s)\xi,\tau\bigr)-u\bigl((c+s)\xi,\tau\bigr)\Bigr)d\tau\\&\quad+\int_0^t u\bigl((c+s)\xi,\tau\bigr)\Bigl(\beta- (Gu)\bigl((c+s)\xi,\tau\bigr)\Bigr)d\tau.
\end{align*}
Therefore, for $r>c$,
\begin{align}\notag
0&\geq \ka\int_{-r}^r\int_0^t \Bigl((a*u)\bigl((c+s)\xi,\tau\bigr)-u\bigl((c+s)\xi,\tau\bigr)\Bigr)d\tau\,ds\\&\quad+\int_{-r}^r\int_0^t u\bigl((c+s)\xi,\tau\bigr)\Bigl(\beta- (Gu)\bigl((c+s)\xi,\tau\bigr)\Bigr)d\tau\,ds\notag\\&\quad 
+\int_{-r}^r \Bigl(u_0\bigl((c+s)\xi\bigr)-u_0(s\xi)\Bigr)\,ds=:S_1(r) +S_2(r) + S_3(r).\label{eq:qwqwqwrwrq}
\end{align}

 Note that $u_0$ is constant along any $\eta\in\S$ orthogonal to $\xi$, cf.~Definition~\ref{def:monotoneindirection}; and, by Proposition~\ref{prop:monot_along_vector_sol},
 $u$ has the same property. Namely, for each $s\in\R$ and $\eta\in\S$ orthogonal to $\xi$,
\begin{equation}\label{eq:constalongeta}
	u(x,t) = u(x+s\eta,t),\qquad t\geq0,\ x\in \X.
\end{equation}
For $d\geq2$, choose any $\{\eta_{1},\ \eta_{2},\ ...,\ \eta_{d-1}\}\subset\S $ which form a complement of $\xi\in\S $ to an orthonormal basis in $\X $.
Then 
\begin{align}
	&\quad (a*u)(s\xi,t)=\int_\X a(y)u(s\xi-y,t)dy\notag\\
	&=\int_\X a\biggl(\,\sum_{j=1}^{d-1}y_j\eta_j+y_d\xi\biggr)\,
    u\biggl(-\sum_{j=1}^{d-1}y_j\eta_j+(s-y_d)\xi,t\biggr)\,dy_1\ldots dy_d\notag\\
&=\int_\R\Biggl(\int_{\R^{d-1}}a\biggl(\,\sum_{j=1}^{d-1}y_j\eta_j+y_d\xi\biggr)\,
dy_1\ldots dy_{d-1}\Biggr)u\bigl((s-y_d)\xi,t\bigr)\,dy_d,\label{reducingto1dim}
\end{align}
where we used \eqref{eq:constalongeta} with $\eta=-\sum\limits_{j=1}^{d-1}y_j\eta_j$, which is orthogonal to the $\xi$.
Therefore, one can set
\begin{equation*}
\check{a}(s):=\begin{cases}
	\displaystyle \int_{\R^{d-1}} a\biggl(\,\sum_{j=1}^{d-1}y_j\eta_j+s\xi\biggr)\,dy_1\ldots dy_{d-1}, &d\geq2,\\[3mm]
a^\pm(s\xi), &d=1
\end{cases}
\end{equation*}
for $s\in\R$. We also denote $\check{u}(s,t):= u(s\xi,t)$, $s\in\R$.
Then one can continue \eqref{reducingto1dim}, as follows: $(a*u)(s\xi,t)=(\check{a}*\check{u})(s,t)$, where the convolution in the right-hand side is in $s\in\R$.
Since $\int_\R \check{a}(s)ds=\int_\X  a(y)dy =1$ and \eqref{assum:first_moment_finite} yields 
\[
	\int_\R |s|\check{a}(s)ds=\int_\X |y\cdot \xi| a(y)dy   <\infty,
\]
we may apply Lemma~\ref{lem:average_of_jump_gen_is_zero} with $b = \check{a}$ and  $v$ given by \eqref{eq:defofvadd}. Then, by \eqref{eq:average_of_jump_gen_is_zero}, \eqref{eq:limit_of_f} and the dominated convergence theorem, we have
\begin{align}
	S_1(r) &= \ka \int_0^t \int_{-r}^r  \bigl( (\check{a}*\check{u})(s+c,\tau) - \check{u}(s+c,\tau) \bigr) dsd\tau \notag \\
	&\qquad \to \ka t \theta \int_\R s \check{a}(s)ds = \ka  \theta t \int_\X  y\!\cdot\! \xi a(y)dy =\ka  \theta t\, \xi\!\cdot\!\m, \label{eq:S_1_lim}
\end{align}
as $r\to\infty$. 
Next, by \eqref{eq:limit_of_f}, \eqref{eq:setc}, we have, cf. \eqref{eq:monimplconv}, 
\begin{equation}\label{eq:S_2_lim}
	S_3(r) = \int_r^{r+c} u_0(s\xi) ds - \int_{-r}^{-r+c} u_0(s\xi)ds \to -\theta c = -\theta \ka t \, \xi\!\cdot\!\m ,
\end{equation}
as $r\to\infty$. 
Therefore, combining \eqref{eq:qwqwqwrwrq}, \eqref{eq:S_1_lim}, \eqref{eq:S_2_lim} with the inequality $u(\beta-Gu)\geq 0$, we deduce that 
\begin{equation}\label{eq:S_3_lim}
	 \int_{-\infty}^\infty \int_0^tu\bigl((c+s)\xi,\tau\bigr)\Bigl(\beta- (Gu)\bigl((c+s)\xi,\tau\bigr)\Bigr)d\tau\,ds=\lim_{r\to\infty} S_2(r) =  0.
\end{equation}
Let $w_0\in \Cb$ be such that $0\leq w_0\leq u_0 $ and $w_0\not\equiv 0$. The by Theorem~\ref{thm:compar_pr} and Proposition~\ref{prop:u_gr_0}, we have
\[
	u(x,\tau) \geq w(x,\tau) > 0, \quad x\in\X, \ \tau>0.
\]
Hence \eqref{eq:S_3_lim} is possible if and only if $(Gu)\bigl(s\xi,\tau\bigr)=\beta$ for (a.a.) $s\in\R$ and all $\tau\in[0,t]$; note that $u(\cdot,\tau)$ is continuous in $\tau\geq0$ and $G$ is continuous on $E_\theta^+$ because of~\eqref{assum:Glipschitz}. In particular,  $(Gu_0)\bigl(s\xi\bigr)=\beta$, $s\in\R$. Then we have by \eqref{assum:G_commute_T} that, for any $p>0$, 
\begin{equation}\label{eq:togetcontradiction}
(GT_{-p\xi} u_0)(s\xi)=(T_{-p\xi}G u_0)(s\xi)=(Gu_0)((s+p)\xi)=\beta, \quad s\in\R.
\end{equation}
Since, $(T_{-p\xi} u_0)(x)=f(x\!\,\cdot\,\!\xi+p)$, $x\in\X$, and $f(\infty)=0$, we have that $T_{-p\xi} u_0 \locun 0$, as $p\to\infty$. Then, by \eqref{assum:G_locally_continuous}, \eqref{assum:Gpositive} we get that $GT_{-p\xi} u_0\locun G0=0$, as $p\to\infty$, that contradicts \eqref{eq:togetcontradiction}. The proof is fulfilled.
\end{proof}

Therefore, under assumptions \eqref{assum:kappa>m}--\eqref{assum:first_moment_finite}, one has that \eqref{eq:upsisnotemp} holds for all $T>0$ and, moreover,
\eqref{hyp:descrete_front_nonempty} holds for $\n=\m$ given by \eqref{firstfullmoment}. Now, we are going to get rid of the condition \eqref{eq:init_cond_is_large}.

We find first a useful sub-solution to the linearization of \eqref{eq:basic} around the zero solution, namely
\begin{equation}\label{eq:linear}
	\dfrac{\partial v}{\partial t}(x,t) = \ka (a*v)(x,t) -mv(x,t).
\end{equation}
\begin{prop}\label{prop:subsolutiontolinear}
	Let \eqref{assum:kappa>m}, \eqref{assum:a_nodeg}, \eqref{assum:first_moment_finite} hold  and $\m$ be given by \eqref{firstfullmoment}.
  Then there exists $\alpha_0>0$, such that, for all $\alpha\in(0,\alpha_0)$, there exists $T=T(\alpha)>0$, such that, for all $q>0$, the function
   \begin{equation}\label{greatsubsol}
     w(x,t)=q\exp\biggl(-\frac{|x- t \m   |^2}{\alpha t}\biggr), \quad x\in\X,\  t>T,
\end{equation}
is a sub-solution to \eqref{eq:linear} on $t>T$; i.e., cf. \eqref{Foper}, 
\begin{equation}\label{eq:subsollineshow}
(\widetilde{\mathcal{F}} w)(x,t):=\frac{\partial w(x,t)}{\partial t} - \ka (a*w)(x,t) +mw(x,t)\leq 0
\end{equation}
for all $x\in\X$, $t>T$.
\end{prop}
The proof is very similar to that in \cite[Proposition 5.19]{FKT2015}.
For reader convenience, we provide the proof in the Appendix.

Now, we will show that \eqref{greatsubsol} is a sub-solution to \eqref{eq:basic} provided that $q$ is small enough.
\begin{prop}\label{prop:subsolution}
	Let \eqref{assum:kappa>m}--\eqref{assum:first_moment_finite} hold  and $\m$ be given by \eqref{firstfullmoment}.
  Then there exists $q_0\in(0,\theta)$ and  $\alpha_0>0$, such that, for all $\alpha\in(0,\alpha_0)$, there exists
$T=T(\alpha)>0$, such that, for all $q\in(0,q_0)$, the function \eqref{greatsubsol} 
is a sub-solution to \eqref{eq:basic} on $t>T$; i.e., cf. \eqref{Foper} and \eqref{eq:subsollineshow}, 
\[
(\mathcal{F}_\theta w)(x,t):=\frac{\partial w(x,t)}{\partial t} - \ka (a*w)(x,t) +mw(x,t) +w(x,t)(Gw)(x,t)\leq 0
\]
for all $x\in\X$, $t>T$.
\end{prop}
\begin{proof}
By \eqref{assum:Gpositive}, \eqref{assum:Glipschitz}, for each $0<q_0<\min\bigl\{\theta, \frac{\beta}{2l_\theta}\bigr\}$ (where, recall, $\beta=\ka-m$), we have that $v\in E_{q_0}^+$ yields $0\leq Gv \leq \frac{\beta}{2}$ .
	Then, for each $q\in(0,q_0)$,
	\[
			\mathcal{F}_\theta w \leq \frac{\partial w}{\partial t} - \ka a*w +\Bigl(m+\frac{\beta}{2}\Bigr)w.
	\]
	Since \eqref{assum:kappa>m} yields $m+\frac{\beta}{2}<\ka$, the statement follows from Proposition~\ref{prop:subsolutiontolinear} applied for \eqref{eq:subsollineshow} with $m$ replaced by $m+\frac{\beta}{2}$.
\end{proof}

The next statement shows that a solution to \eqref{eq:basic} becomes larger than the sub-solution \eqref{greatsubsol} after a big enough time.

\begin{prop}\label{prop:useBrandle}
	Let \eqref{assum:kappa>m}--\eqref{assum:improved_sufficient_for_comparison} hold.
Then, there exists $t_1>0$, such that, for any $t>t_1$ and for any $\tau>0$, there exists $q_1=q_1(t,\tau)>0$, such that the following holds.
If $u_0\in E^+_\theta$ is such that there exist $\eta>0$, $r >0$, $x_0\in\X$ with $u_0(x)\geq\eta$, $x\in B_r (x_0)$ and $u$ is the corresponding solution to \eqref{eq:basic}, then
  	\[
	  	u(x,t)\geq q_1\exp\Bigl(-\frac{|x-x_0|^2}{\tau}\Bigr),\quad x\in\X.
  	\]
\end{prop} 
The proof is, as a matter of fact, the same as that in  \cite[Proposition 5.20]{FKT2015}.	Again, for reader convenience, we provide the proof in the Appendix.

Now we are finally ready to proof Theorems~\ref{thm:ht1myes}, \ref{thm:ht1mno}.

\begin{proof}[Proof of Theorem~\ref{thm:ht1myes}]
As it was mentioned above, one can get the statement, combining Propositions~\ref{prop:hairtrigger_general} and \ref{prop:suff_H1}, provided that \eqref{eq:init_cond_is_large} holds. To get rid of the latter assumption, one can literally follow the proof of \cite[Theorem 5.10]{FKT2015} using the results of 
 Propositions~\ref{prop:subsolution} and~\ref{prop:useBrandle}. 
\end{proof}

\begin{proof}[Proof of Theorem~\ref{thm:ht1mno}]
	Without loss of generality we can assume that $\theta-\theta_n\leq \frac{\theta}{2}$, $n\in\N$. Consider $v_0\in E_{\theta\! /\! 2}^+\cap C^\infty(\X)$, such that  for some $x_0\in \X$, $\delta\in(0,\frac{\theta}{2})$,
\[
	\delta \1_{B_\delta(x_0)}(x) \leq v_0(x) \leq u_0(x),\qquad x\in\X.
\]
Let $u_n(x,0)=v_0(x)$ and $u_n$ solves the following equation
\begin{equation}\label{eq:unbasic}
	\mathcal{F}_{\theta_n}^{(n)} u_n := \frac{\partial u_n}{\partial t} - \ka_n a_n*u_n +u_n G_n u_n + mu_n = 0.
\end{equation}
Therefore by \eqref{assum:approx_of_basic} we obtain, 
\[
		\mathcal{F}_{\theta_n}^{(n)} u_n= 0 = \mathcal{F}_\theta u \leq \mathcal{F}_{\theta_n}^{(n)} u.
\]
Hence by Theorem~\ref{thm:compar_pr} applied to $\mathcal{F}_{\theta_n}^{(n)}$, we obtain
\[
	0 \leq u_n(x,t) \leq u(x,t) \leq \theta.
\]
Applying Theorem~\ref{thm:ht1myes} to the equation~\eqref{eq:unbasic}, we have
\begin{align*}
	\theta-\frac{1}{n} \leq \theta_n =&  \lim_{t\to\infty} \essinf_{x\in K} u_n(x+t\m_n,t)   \leq \liminf_{t\to\infty} \essinf_{x\in K} u(x+t\m_n,t)\\ 
 \leq& \limsup_{t\to\infty} \essinf_{x\in K} u(x+t\m_n,t) \leq \theta,
\end{align*}
that fulfills the proof.
\end{proof}

\section*{Acknowledgments}
Authors gratefully acknowledge the financial support by the DFG through CRC 701 ``Stochastic
Dynamics: Mathematical Theory and Applications'' (DF and PT), the European
Commission under the project STREVCOMS PIRSES-2013-612669 (DF), and the ``Bielefeld Young Researchers'' Fund through the Funding Line Postdocs: ``Career Bridge Doctorate\,--\,Postdoc'' (PT).

\section*{Appendix}
\begin{proof}[Proof of Proposition~\ref{prop:subsolutiontolinear}]	
For $q>0$, consider the function \eqref{greatsubsol}. By \eqref{eq:subsollineshow}, one gets
\[
 (\widetilde{\mathcal{F}}w)(x,t)=w(x,t)\biggl(\frac{|x|^2}{\alpha t^2}-\frac{|\m   |^2}{\alpha}\biggr) - \ka (a*w)(x,t)+mw(x,t).
\]
Therefore, to have $\widetilde{\mathcal{F}}w\leq 0$, it is enough to claim that, for all $x\in\X$,
\[
m+\frac{|x|^2}{\alpha t^2}-\frac{|\m   |^2}{\alpha}\leq
\ka \exp\biggl(\frac{|x-t \m   |^2}{\alpha t}\biggr)
\int_\X a(y)\exp\biggl(-\frac{|x-y-t \m   |^2}{\alpha t}\biggr)\,dy.
\]
By changing $x$ onto $x+t \m   $ and a simplification, one gets an equivalent inequality
\begin{equation}\label{ineqtoprove}
 m+ \frac{|x|^2}{\alpha t^2}+\frac{2\, x\cdot\m   }{\alpha t}\leq\ka 
\int_\X a(y)\exp\biggl(\frac{2x\cdot y}{\alpha t}\biggr)\exp\biggl(-\frac{|y|^2}{\alpha t}\biggr)\,dy=:I(t).
\end{equation}
One can rewrite $I(t)=I_0(t)+I^+(t)+I^-(t)$, where
\begin{gather*}
I_0(t):=\ka \int_\X a(y)e^{-\frac{|y|^2}{\alpha t}}dy;\qquad\quad
I^+(t):=\ka \int_{x\cdot y\geq0} a(y)e^{-\frac{|y|^2}{\alpha t}}\Bigl(e^{\frac{2x\cdot y}{\alpha t}}-1\Bigr)dy;\\
I^-(t):=\ka \int_{x\cdot y<0} a(y)e^{-\frac{|y|^2}{\alpha t}}\Bigl(e^{\frac{2x\cdot y}{\alpha t}}-1\Bigr)dy.
\end{gather*}
Using that $e^s-1\geq s$, for all $s\in\R$, and $e^s-1\geq s+\frac{s^2}{2}$, for all $s\geq0$, one gets the following estimates
\begin{align*}
I^+(t)&\geq  \frac{2\ka}{\alpha t}\int_{x\cdot y\geq0} a(y)e^{-\frac{|y|^2}{\alpha t}}(x\cdot y)dy+ \frac{2\ka}{\alpha^2 t^2}\int_{x\cdot y\geq0} a(y)e^{-\frac{|y|^2}{\alpha t}}(x\cdot y)^2dy,\\
I^-(t)&\geq\frac{2\ka}{\alpha t}\int_{x\cdot y<0} a(y)e^{-\frac{|y|^2}{\alpha t}}(x\cdot y)dy.
\end{align*}
Therefore,
\begin{align}
I(t)&\geq I_0(t) + \frac{2}{\alpha t}\, x\cdot I_1(t) +\frac{2}{\alpha^2 t^2} I_2(t),\label{estIfrombelow}\\
\shortintertext{where}
I_1(t)&:=\ka\int_{\X} a(y)e^{-\frac{|y|^2}{\alpha t}} y \,dy\in\X,\notag\\
I_2(t)&:=\ka\int_{x\cdot y\geq0} a(y)e^{-\frac{|y|^2}{\alpha t}}(x\cdot y)^2dy\in\R.\notag
\end{align}

By \eqref{assum:first_moment_finite}, \eqref{firstfullmoment}, and the dominated convergence theorem, 
we will get that $I_0(t)\nearrow \ka >m$ and  $I_1(t)\to\m \in\X$ as $t\to\infty$. 
Therefore, for any $\eps>0$ with $m+2\eps<\ka$,  there exists $T_1=T_1(\eps)>0$, such that, for all $\alpha>0$ and $t>0$ with $\alpha t>T_1$, one has
\begin{equation}\label{estlim}
\ka\geq I_0(t)>m+\eps, \qquad \lvert I_1(t)-\m   \rvert<\eps.
\end{equation}

Let $T>\frac{T_1}{\alpha}$ be chosen later. The function $I_2(t)$ is also increasing in $t>0$.
Therefore, by \eqref{estIfrombelow} and \eqref{estlim}, one gets, for $t>T>\frac{T_1}{\alpha}$,
\begin{align}
I(t)&>m+\eps +\frac{2}{\alpha t}x\cdot(I_1(t)-\m   )+\frac{2}{\alpha t}x\cdot\m
+\frac{2}{\alpha^2 t^2}I_2(t) \notag\\
& \geq m+\eps -\frac{2\eps}{\alpha t}|x|+\frac{2}{\alpha t}x\cdot\m
+\frac{2}{\alpha^2 t^2}I_2(T).\label{ineqenoughtoprove2}
\end{align}

Let $\varrho>0$ be as in \eqref{assum:a_nodeg}. For an arbitrary $x\in\X$, consider the set
\[
B_x=\Bigl\{y\in\X \Bigm\vert |y|\leq \varrho, \frac{1}{2}\leq\frac{x\cdot y}{|x||y|}\leq 1\Bigr\}.
\]
Then
\begin{equation}\label{weneedit}
I_2(T)\geq \frac{\ka\varrho}{4}|x|^2 \int_{B_x} |y|^2 e^{-\frac{|y|^2}{\alpha T}} dy.
\end{equation}
The set $B_x$ is a cone inside the ball $B_\varrho(0)$, with the apex at the origin, the height which lies along $x$, and the apex angle $2\pi/3$. Since the function inside the integral in the right-hand side of \eqref{weneedit} is radially symmetric, the integral does not depend on $x$. Fix an arbitrary $\bar{x}\in\X$ and denote
\begin{equation}\label{Alimit}
  A(\tau)=A(\tau,\varrho)=\int_{B_{\bar{x}}} |y|^2 e^{-\frac{|y|^2}{\tau}} dy\nearrow\int_{B_{\bar{x}}} |y|^2 dy=:\bar B_\varrho, \quad \tau\to\infty.
\end{equation}
Then, by \eqref{ineqenoughtoprove2} and \eqref{weneedit}, one has, for $t>T$,
\begin{equation}
I(t)> m+\eps -\frac{2\eps}{\alpha t}|x|+\frac{2}{\alpha t}x\cdot\m
+\frac{\ka\varrho A(\alpha T)}{2\alpha^2 t^2}|x|^2 .\label{ineqenoughtoprove3}
\end{equation}
By \eqref{ineqenoughtoprove3}, to prove \eqref{ineqtoprove}, it is enough to show that
\begin{equation*}
  \eps -\frac{2\eps}{\alpha t}|x|
+\frac{\ka\varrho A(\alpha T)}{2\alpha^2 t^2}|x|^2\geq  \frac{|x|^2}{\alpha t^2}, \qquad t>T, \ x\in\X,
\end{equation*}
or, equivalently, for $2\alpha<\ka\varrho  A(\alpha T)$,
\begin{multline}\label{ineqenoughtoprove}
\biggl(\sqrt{\frac{\ka\varrho A(\alpha T)-2\alpha}{2 }}\frac{|x|}{\alpha t}-
 \eps\sqrt{\frac{2}{\ka\varrho A(\alpha T)-2\alpha}}\biggr)^2\\+\eps-\eps^2\frac{2}{\ka\varrho A(\alpha T)-2\alpha}\geq0.
\end{multline}
To get \eqref{ineqenoughtoprove}, we proceed as follows. For a given $\varrho>0$ which provides \eqref{assum:a_nodeg}, we set $\alpha_0:=\frac{1}{2}\ka\varrho\bar{B}_\varrho$, cf.~\eqref{Alimit}.
Then, for any $\alpha\in(0,\alpha_0)$, there exists $T_2=T_2(\alpha)>0$, such that
\[
2\alpha<\ka\varrho A(\alpha T_2)<\ka\varrho\bar{B}_\delta.
\]
Choose now $\eps=\eps(\alpha)>0$, such that $m+2\eps<\ka$ and
\begin{equation}\label{smalleps}
  \eps<\frac{1}{2}(\ka\varrho A(\alpha T_2)-2\alpha)<
\frac{1}{2}(\ka\varrho A(\alpha T)-2\alpha), \quad T>T_2.
\end{equation}
Then, find $T_1=T_1(\alpha)>0$ which gives \eqref{estlim} for $\alpha t>T_1$; and, finally, take $T=T(\alpha)>T_2$ such that $\alpha T> T_1$. As a result, for $t>T$, one has $\alpha t>\alpha T>T_1$, thus
\eqref{estlim} holds, whereas \eqref{smalleps} yields \eqref{ineqenoughtoprove}. The latter inequality gives \eqref{ineqtoprove}, and hence, for all $q>0$, $\mathcal{F}w\leq 0$, for $w$ given by \eqref{greatsubsol}. The statement is proved.
\end{proof}

	\begin{proof}[Proof of Proposition~\ref{prop:useBrandle}]
	By (Q2) in Proposition \ref{prop:Q_def}, it is enough to prove the statement for $x_0=0$.
	Consider arbitrary functions $j,v_0\in C^\infty(\X)$, such that
 	\begin{align*}
 		\supp j =B_\delta(0), &\qquad 0 < j(x)=j(|x|) \leq \delta, && x\in \inter(B_\delta(0));\\
 		\supp v_0=B_r(0), &\qquad 0 < v_0(x) \leq \eta, && x\in \inter(B_r(0));\\
 		\exists \, 0<p<\min\{r,1\},\, 0<\nu<\eta, &\qquad \text{such that} \ v_0(x)\geq \nu, && x\in B_{p }(0),
 	\end{align*}
where $\delta$ is the same as in \eqref{assum:improved_sufficient_for_comparison}.
We choose $p$ and $b$ as in \eqref{assum:improved_sufficient_for_comparison}. Then one can rewrite \eqref{eq:basic} as follows
	\begin{align*}
		\frac{\partial}{\partial t} u(x,t) = \ka (j*u)(x,t)-(m+q)u(x,t)+f(x,t),
	\end{align*}
	where, for all $x\in\X$ and $t\geq0$,
	\begin{align*}
		f(x,t)&:= \ka ((a-j)*u)(x,t)-u(x,t)(Gu)(x,t) + qu(x,t) \geq 0,
	\end{align*}
	because of \eqref{assum:improved_sufficient_for_comparison}.
	Since $j\geq0$ and $Ju=j*u$ defines a bounded operator on $L^\infty(\X)$, one has that $e^{t J}f(x,s)\geq0$, for all $t,s\geq0$, $x\in\X$.
	By the same argument, $u_0(x)\geq \eta \1_{B_r(0)}(x)\geq v_0(x)\geq0$ implies
$(e^{t J}u_0)(x)\geq (e^{t J}v_0)(x)$.
Therefore,
\begin{align}\notag
u(x,t)&=e^{-t(m+q)}(e^{tJ}u_0)(x)+\int_0^t e^{-(t-s)(m+q)}(e^{(t-s)J}f)(x,s)ds\\
&\geq e^{-t(m+q)}(e^{t J}u_0)(x)
\geq e^{-(m+q-\langle j\rangle)t}(e^{tL_j}v_0)(x),\quad x\in\X.\label{estbelow1}
\end{align}
where $\langle j\rangle:=\int_\X j(x)\,dx>0$ and $L_j u=Ju-\langle j\rangle u$.

We are going to apply now the results of \cite{BCF2011}. To do this, set $\alpha:=\langle j\rangle ^{-1}$. Then,
\begin{equation}\label{asfsaffsa}
(e^{tL_j}v_0)(x)=(e^{\langle j\rangle t(\alpha L_j)}v_0)(x)=v(x,\langle j\rangle t),
\end{equation}
where $v$ solves the differential equation $\frac{d}{dt}v=\alpha L_j$.
Since $\int_\X \alpha j(x)\,dx=1$, then, by \cite[Theorem~1.4, Lemma~1.6]{AMRT2010},
\begin{equation}\label{soltolinear}
  v(x,t)=e^{-t}v_0(x)+(w*v_0)(x,t),
\end{equation}
where $w(x,t)$ is a smooth function. Moreover, by \cite[Proposition 5.1]{BCF2011}, for any $\omega\in (0,\delta)$ there exist $c_1=c_1(\omega)>0$ and $c_2=c_2(\omega)\in\R$, such that
\begin{equation}\label{lowbound}
\begin{split}
w(x,t)&\geq h(x,t), \quad x\in\X, t\geq0,\\
h(x,t)&:=c_1t\exp\Bigl(-t-\frac{1}{\omega}|x|\log|x|+(\log t-c_2)\Bigl[\frac{|x|}{\omega}\Bigr]\Bigr).
\end{split}
\end{equation}
Here $[\alpha]$ means the entire part of an $\alpha\in\R$, and $0\log 0:=1$, $\log 0:=-\infty$.

Set $t_1=e^{c_2}>0$. Since $[\alpha]>\alpha-1$, $\alpha\in\R$, one has, for $t>t_1$,
\[
h(x,t)\geq c_1e^{c_2}\exp\Bigl(-t-\frac{1}{\omega}|x|\log|x|+(\log t-c_2)\frac{|x|}{\omega}\Bigr)\geq c_3g(x,t),
\]
where $c_3=c_1e^{c_2}>0$ and
\[
g(x,t):=\exp\Bigl(-t-\frac{1}{\omega}|x|\log|x|\Bigr), \quad x\in\X, t>t_1.
\]
Since $v_0\geq\nu\1_{B_{p }(0)}$, one gets from \eqref{soltolinear} and \eqref{lowbound}, that
\begin{equation}\label{aswgddsye}
   v(x,t)\geq \nu e^{-t}\1_{B_{p }(0)}(x)+\nu c_3\int_{B_{p }(x)}g(y,t)\,dy
\end{equation}
Set $V_p :=\int_{B_p(0)}\,dx$. For any fixed $t>t_1$, since $g(\cdot,t)\in C(B_p(x))$, there exists $y_0,y_1\in B_p(x)$, such that $g(y,t)$ attains its minimal and maximal values on $B_p(x)$ at these points, respectively. Since $B_p(x)$ is a convex set, one gets that, for any $\gamma\in(0,1)$, $y_\gamma:=\gamma y_1+(1-\gamma)y_0\in B_p(x)$. Then
\[
V_p g(y_0,t)\leq \int_{B_{p }(x)}g(y_\gamma,t)\,dy\leq V_p g(y_1,t).
\]
Therefore, by the intermediate value theorem there exists, $\tilde{y}_t=\tilde{y}(x,t)\in B_p (x)$, $t>t_1$, $x\in\X$, such that $\int_{B_p(x)}g(y,t)\,dy=V_p  g(\tilde{y}_t,t)$. Hence one gets from \eqref{estbelow1}, \eqref{asfsaffsa}, \eqref{aswgddsye}, that
  \begin{align}
  u(x,t)&\geq c_4 e^{-(m+q-\langle j\rangle)t}g\bigl(\tilde{y}_t,\langle j\rangle t\bigr)\notag\\
  &=c_4
  \exp\Bigl(-(m+q)t-\frac{1}{\omega}|\tilde{y}_t|\log|\tilde{y}_t|\Bigr),\label{estbelow2}
  \end{align}
for $\tilde{y}_t=\tilde{y}(x,t)\in B_p (x)$, $t>t_1$; here $c_4=c_3 \nu V_p>0$.

As a result, to get the statement, it is enough to show that, for any $t>t_1$ and for any $\tau>0$, there exists $q_1=q_1(t,\tau)>0$, such that the r.h.s. of \eqref{estbelow2} is estimated from below by $q_1 e^{-\frac{|x|^2}{\tau}}$, i.e. that
\begin{equation}\label{needtoproveineq2}
(m+q)t+\frac{1}{\omega}|\tilde{y}_t|\log|\tilde{y}_t|-\log c_4\leq \frac{|x|^2}{\tau}-\log q_1, \quad x\in\X,
\end{equation}
Note that $\tilde{y}_t\in B_p (x)$ implies $|\tilde{y}_t|\leq p+|x|$, $x\in\X$.

Let $p+|x|\leq 1$. Then $\log |\tilde{y}_t|\leq 0$, and the l.h.s. of \eqref{needtoproveineq2} is majorized by $(m+q)t-\log c_4$. Therefore, to get \eqref{needtoproveineq2}, it is enough to have $q_1< c_4e^{-(m+q)t}$, regardless~of~$\tau$.

Let now $|x|+p >1$. Recall that we chose $p<1$. The function $s\log s$ is increasing on $s>1$. Hence to get \eqref{needtoproveineq2}, we claim
\begin{equation}\label{sasfaasftewtew}
  (|x|+1)\log(|x|+1)\leq \frac{\omega}{\tau}|x|^2-\omega (m+q)t+\omega\log c_4-\omega\log q_1.
\end{equation}
Consider now the function $f(s)=as^2-(s+1)\log(s+1)$, $s\geq0$, $a=\frac{\omega}{\tau}>0$. Then $f(0)=0$, $f'(s)=2as-\log(s+1)-1$, $f'(0)=-1$, $f''(s)=2a-\frac{1}{s+1}$. Since $f''(s)\nearrow 2a>0$, $s\to\infty$, there exists $s_0>0$, such that $f''(s)>0$, for all $s>s_0$, i.e. $f'(s)$ increases on $s>s_0$. Since $f'(s)\to\infty$, $s\to\infty$, there exists $s_1>s_0$, such that $f'(s)>0$, for all $s>s_1$, i.e. $f$ is increasing on $s>s_1$. Finally, for any $t>t_1$, one can choose $q_1=q_1(t,\tau)>0$ small enough, to get
\[
\min\limits_{s\in[0,s_1]}f(s)-\omega (m+q)t+\omega\log c_4-\omega\log q_1>0
\]
and to fulfill \eqref{sasfaasftewtew}, for all $x\in\X$.
The statement is proved.
\end{proof}

\def\cprime{$'$}

\end{document}